\documentclass[12pt]{article}
\title{Dp-finite fields II: the canonical topology and its relation to
  henselianity}
\author{Will Johnson}

\usepackage{amsmath, amssymb, amsthm}    	
\usepackage{fullpage} 	
\usepackage{amscd}
\usepackage{hyperref}
\usepackage[all]{xy}
\usepackage{centernot}

\DeclareMathOperator*{\forkindep}{\raise0.2ex\hbox{\ooalign{\hidewidth$\vert$\hidewidth\cr\raise-0.9ex\hbox{$\smile$}}}}

\newcommand{\Gal}{\operatorname{Gal}}

\newcommand{\Frac}{\operatorname{Frac}}

\newcommand{\characteristic}{\operatorname{char}}

\newcommand{\Tr}{\operatorname{Tr}}
\newcommand{\res}{\operatorname{res}}
\newcommand{\Aut}{\operatorname{Aut}}

\newcommand{\dcl}{\operatorname{dcl}}
\newcommand{\tp}{\operatorname{tp}}

\newcommand{\val}{\operatorname{val}}

\newcommand{\dpr}{\operatorname{dp-rk}}
\newcommand{\redrk}{\operatorname{rk}_0}

\newcommand{\mininf}{-_\infty}

\newtheorem{theorem}{Theorem}[section] 
\newtheorem{lemma}[theorem]{Lemma}

\newtheorem{corollary}[theorem]{Corollary}

\newtheorem{conjecture}[theorem]{Conjecture}

\newtheorem{proposition}[theorem]{Proposition}
\newtheorem*{theorem-star}{Theorem}
\newtheorem*{conjecture-star}{Conjecture}

\theoremstyle{definition}
\newtheorem{definition}[theorem]{Definition}
\newtheorem{example}[theorem]{Example}

\theoremstyle{remark}
\newtheorem{remark}[theorem]{Remark}
\newtheorem{claim}[theorem]{Claim}

\newtheorem*{acknowledgment}{Acknowledgments}

\newcommand{\Rr}{\mathbb{R}}

\newcommand{\Zz}{\mathbb{Z}}
\newcommand{\Kk}{\mathbb{K}}

\newcommand{\Mm}{\mathbb{M}}

\newcommand{\Ll}{\mathbb{L}}
\newcommand{\Oo}{\mathcal{O}}
\newcommand{\mm}{\mathfrak{m}}

\newenvironment{claimproof}[1][\proofname]
               {
                 \proof[#1]
                 
               }
               {
                 \endproof
               }

\begin{document}
\maketitle

\begin{abstract}
  We continue the investigation of dp-finite fields from \cite{prdf}.
  We show that the ``heavy sets'' of \cite{prdf} are exactly the sets
  of full dp-rank.  As a consequence, full dp-rank is a definable property
  in definable families of sets.  If $I$ is the group of
  infinitesimals, we show that $1 + I$ is the group of multiplicative
  infinitesimals.  From this, we deduce that the canonical topology is
  a field topology.  Lastly, we consider the (unlikely) conjecture
  that the canonical topology is a V-topology.  Assuming this
  conjecture, we deduce the expected classification of dp-finite
  fields \cite{halevi-hasson-jahnke}.
\end{abstract}

\section{Introduction}\label{sec:1}
This paper continues the investigation in \cite{prdf} of fields of
finite dp-rank, also known as \emph{dp-finite fields}.  For background
on dp-rank, see \cite{dp-add}.  In \cite{arxiv-myself}, dp-minimal
fields were classified up to elementary equivalence in the language of
rings.  One hopes to generalize the proof to dp-finite fields, or
even strongly dependent or NIP fields.

By work of Halevi, Hasson, Jahnke, Koenigsmann, Sinclair, and others
(\cite{halevi-hasson-jahnke,hhj-v-top,JK,sinclair}), the
classification of dp-finite fields is known, modulo the following
conjectures.  See \cite{halevi-hasson-jahnke} for details.
\begin{conjecture-star}[Shelah conjecture for dp-finite fields]
  Let $K$ be a dp-finite field.  Then one of the following holds:
  \begin{itemize}
  \item $K$ is finite
  \item $K$ is algebraically closed
  \item $K$ is real closed
  \item $K$ admits a non-trivial henselian valuation
  \end{itemize}
\end{conjecture-star}
\begin{conjecture-star}[henselianity conjecture for dp-finite fields]
  Let $(K,v)$ be a dp-finite valued field.  Then $v$ is henselian.
\end{conjecture-star}
These conjectures were proven for $K$ of positive characteristic in
\cite{prdf} (Corollary~11.4 and Theorem~2.8).  In the present paper,
we outline a strategy for handling characteristic 0, based on the
original strategy for dp-minimal fields \cite{arxiv-myself}.

\begin{remark}
  The Shelah conjecture in fact implies the henselianity conjecture
  \cite{hhj-v-top}.  But the known methods for attacking the Shelah
  conjecture go in the other direction: one produces a valuation,
  proves the henselianity conjecture, and deduces the Shelah
  conjecture as a consequence.
\end{remark}

\subsection{New results on the canonical topology}
Let $K$ be an unstable dp-finite field.  In \cite{prdf},
Definition~4.19, we defined a notion of \emph{heaviness} for definable
subsets $X \subseteq K$, and used this to define a ``canonical
topology'' on $K$.  The canonical topology is a Hausdorff non-discrete
ring topology characterized by the fact that the following family of
sets is a neighborhood basis of 0:
\begin{equation*}
  \{X - X : X \subseteq K \text{ definable and heavy}\}.
\end{equation*}
We prove a couple new results which help clarify this picture.
\begin{theorem-star}
  A definable subset $X \subseteq K$ is heavy if and only if $\dpr(X)
  = \dpr(K)$.
\end{theorem-star}
As a corollary, the condition ``$X$ has full rank'' varies definably
with $X$.
\begin{theorem-star}
  The canonical topology on $K$ is a field topology (division is
  continuous).
\end{theorem-star}
This result can be expressed in terms of the infinitesimals.  Let $\Kk
\succeq K$ be saturated, and let $I_K$ be the set of
$K$-infinitesimals:
\begin{equation*}
  I_K = \bigcap \{X - X : X \subseteq \Kk \text{ heavy and $K$-definable}\}
\end{equation*}
The fact that division is continuous is equivalent to the statement
that $1 + I_K$ is a subgroup of $\Kk^\times$.  In fact, we prove
something stronger:
\begin{equation*}
  1 + I_K = \bigcap \{X \cdot X^{-1} : X \subseteq
  \Kk^\times \text{ heavy and $K$-definable}\}
\end{equation*}
Thus $1 + I_K$ is the set of ``multiplicative $K$-infinitesimals,''
which is a subgroup of $\Kk^\times$ for the same reason that $I_K$ is
a subgroup of $\Kk$.

\subsection{Valuation-type fields}
Recall the notion of V-topology from \cite{PZ}.  Say that a
dp-finite field $K$ has \emph{valuation type} if $K$ is unstable and
the canonical topology on $K$ is a V-topology.  Equivalently, the
infinitesimals $I_K$ are the maximal ideal of a valuation ring on
$\Kk$.

This property turns out to be closely related to the Shelah
conjecture.
\begin{theorem-star}
  Let $\Kk$ be a dp-finite field, assumed to be sufficiently
  saturated.  Suppose that for every definable valuation $v$ on $\Kk$,
  the residue field $\Kk v$ is stable or of valuation type.  Then
  \begin{enumerate}
  \item $\Kk$ satisfies the Shelah conjecture.  One of the following
    holds:
    \begin{itemize}
    \item $\Kk$ is finite
    \item $\Kk$ is algebraically closed
    \item $\Kk$ admits a henselian valuation.
    \end{itemize}
  \item Every definable valuation ring on $\Kk$ is henselian.
  \item More generally, every $\vee$-definable valuation ring on $\Kk$
    is henselian.
  \end{enumerate}
\end{theorem-star}
The third point is the heart of the matter, since a $\vee$-definable
valuation ring is given by the assumptions.

We also give a criterion for $K$ to have valuation type.  Say that a
ring $R$ on a field $K$ is a \emph{multi-valuation ring} if it is a
finite intersection of valuation rings on $K$.
\begin{theorem-star}
  Let $\Kk$ be a monster unstable dp-finite field.  The following are
  equivalent:
  \begin{itemize}
  \item $\Kk$ has valuation type.
  \item For every small submodel $K \preceq \Kk$, the infinitesimals
    $I_K$ are the maximal ideal of a valuation ring $\Kk$.
  \item For some small submodel $K \preceq \Kk$, the infinitesimals
    $I_K$ contain a nonzero ideal of a multi-valuation ring on $\Kk$.
  \end{itemize}
\end{theorem-star}
Taken together, these facts suggest that we should attack the Shelah
conjecture and the classification problem by proving the following
conjecture.
\begin{conjecture-star}
  Let $\Kk$ be a monster model of an unstable dp-finite field.  Then
  there is a small submodel $K \preceq \Kk$ for which $I_K$ contains a
  nonzero multi-valuation ideal.
\end{conjecture-star}
This ``multi-valuation strategy'' is the natural generalization of the
method used in \cite{arxiv-myself} to classify dp-minimal fields.  We
discuss the prospects of this strategy in \S\ref{retreat}.
Unfortunately, there is a major problem.

\subsection{Miscellaneous results}
We prove several miscellaneous results concerning dp-finite fields.
\begin{enumerate}
\item The equicharacteristic 0 henselianity conjecture (for NIP
  fields) implies the henselianity conjecture (for NIP fields).  This
  was obvious to experts, but worth pointing out.
\item Sinclair's ``Johnson topology,'' defined in \cite{sinclair} \S
  3.2, \emph{exists}, and is the same as the canonical topology of
  \cite{prdf}.
\item We give another proof that stable fields of finite dp-rank have
  finite Morley rank and are therefore algebraically closed.  This was
  previously proven by Halevi and Palac\'in (\cite{Palacin},
  Proposition 7.2).
\end{enumerate}

\subsection{Notation and conventions}
Unlike \cite{prdf}, the monster will now be $\Kk$.  We will always
assume that $\Kk$ is a field, and never assume that $\Kk$ is a pure
field.

We will sometimes abuse terminology and use ``saturated'' as shorthand
for ``sufficiently saturated and sufficiently strongly homogeneous,''
rather than literally meaning ``saturated in the size of the model.''

\section{Two easy results}\label{sec:2-remarks}
We deduce two easy consequences of \cite{prdf}:
\begin{itemize}
\item Stable fields of finite dp-rank have finite Morley rank.
\item The henselianity conjecture for NIP fields reduces to the case
  of equicharacteristic 0.
\end{itemize}

\subsection{Stable fields of finite dp-rank} \label{sec:stable-case}
Recall that for $A$ a small subset of a monster model $\Mm$, a set is
``$A$-invariant'' if it is $\Aut(\Mm/A)$-invariant.  This is a weaker
notion than being definable, type-definable, or $\vee$-definable over
$A$.
\begin{lemma} \label{stable-valuation-rings}
  Let $\Kk$ be a sufficiently saturated stable field and $A \subseteq
  \Kk$ be small.  There are no non-trivial $A$-invariant valuation
  rings $\Oo$ on $\Kk$.
\end{lemma}
\begin{proof}
  Suppose such an $\Oo$ exists, and let $\val(-)$ denote the
  associated valuation.

  As $\Kk$ is an infinite stable field, there is a global
  $\emptyset$-invariant type $p$ which is the unique generic of both
  the additive group $\Kk$ and the multiplicative group $\Kk^\times$.
  \begin{claim}
    If $(a,b) \models p^{\otimes 2}|A$, then $\val(a) = \val(b)$.
  \end{claim}
  \begin{claimproof}
    By forking symmetry, $(b,a) \models p^{\otimes 2}|A$, so $ab
    \equiv_A ba$.  By $A$-invariance of $\Oo$,
    \begin{equation*}
      \val(a) < \val(b) \iff \val(b) < \val(a),
    \end{equation*}
    implying $\val(a) = \val(b)$.
  \end{claimproof}
  \begin{claim}
    If $a \models p|A$ and $b \models p|A$ then $\val(a) = \val(b)$.
  \end{claim}
  \begin{claimproof}
    Let $c \models p|Aba$.  Then $(a,c)$ and $(b,c)$ both realize
    $p^{\otimes 2}|A$, so $\val(c) = \val(a) = \val(b)$.
  \end{claimproof}
  Now as $\Oo$ is non-trivial, we can find non-zero $\varepsilon \in
  \Kk$ such that $\val(\varepsilon) > 0$.  Let $a \models
  p|A\varepsilon$.  Then $a$ and $a \cdot \varepsilon$ both realize
  $p|A$, so
  \begin{equation*}
    \val(a) = \val(a \cdot \varepsilon),
  \end{equation*}
  contradicting the choice of $\varepsilon$.
\end{proof}  
\begin{theorem}
  Let $\Kk$ be a sufficiently saturated dp-finite field.  Then
  \emph{exactly} one of the following holds:
  \begin{enumerate}
  \item $\Kk$ has finite Morley rank.
  \item $\Kk$ is unstable and there is an $A$-invariant non-trivial
    valuation ring $\Oo$ on $\Kk$, for some small $A \subseteq \Kk$.
  \end{enumerate}
\end{theorem}
\begin{proof}
  Assume $\Kk$ does not have finite Morley rank.  Theorem~10.28 of
  \cite{prdf} proves that the valuation ring exists, and
  Lemma~\ref{stable-valuation-rings} proves that $\Kk$ is unstable.
\end{proof}

This gives another proof that stable fields of finite dp-rank have
finite Morley rank:
\begin{corollary}\label{cor:stable-finite-rank}
  If $\Kk$ is a stable field of finite dp-rank, then $\Kk$ has finite
  Morley rank.
\end{corollary}
\begin{proof}
  We may assume $\Kk$ is sufficiently saturated, in which case the
  theorem applies.
\end{proof}
This was originally proven by Halevi and Palac\'in (\cite{Palacin},
proof of Proposition 7.2 plus the Buechler dichotomy).

\subsection{A remark on the henselianity conjecture} \label{nip-fields-remark}
In \cite{prdf}, Theorem~2.8, we proved that NIP valued fields of positive
characteristic are henselian.  It is conjectured that something
stronger holds:
\begin{conjecture}[Henselianity conjecture] \label{hens-conj}
  If $(K,\Oo)$ is an NIP valued field, then $\Oo$ is henselian.
\end{conjecture}
Consider the more restricted
\begin{conjecture}\label{equi-char}
  If $(K,\Oo)$ is an NIP valued field of residue characteristic 0,
  then $\Oo$ is henselian.
\end{conjecture}
\begin{proposition}
  Conjecture~\ref{equi-char} implies Conjecture~\ref{hens-conj}
\end{proposition}
\begin{proof}
  Let $(K,\Oo)$ be an NIP valued field.  We may assume $(K,\Oo)$ is
  saturated.
  \begin{itemize}
  \item If $(K,\Oo)$ is equicharacteristic 0, use
    Conjecture~\ref{equi-char}.
  \item If $(K,\Oo)$ is equicharacteristic $p$, use \cite{prdf},
    Theorem~2.8.
  \item If $(K,\Oo)$ is mixed characteristic $(0,p)$, proceed as
    follows.  Let $\Gamma$ be the value group.  Let $\Delta^+$ be the
    smallest convex subgroup of $\Gamma$ containing $\val(p)$.  Let
    $\Delta^-$ be the largest convex subgroup of $\Gamma$ not
    containing $\val(p)$.  Let $k$ be the residue field of $(K,\Oo)$.
    Using the convex subgroups, factor the specialization $K
    \rightsquigarrow k$ into three pieces:
    \begin{equation*}
      K \rightsquigarrow A \rightsquigarrow B \rightsquigarrow k
    \end{equation*}
    where
    \begin{enumerate}
    \item $K \rightsquigarrow A$ has value group $\Gamma/\Delta^+$.
    \item $A \rightsquigarrow B$ has value group $\Delta^+/\Delta^-$.
    \item $B \rightsquigarrow k$ has value group $\Delta^-$.
    \end{enumerate}
    The whole picture is NIP because it is interpretable
    in the Shelah expansion: the groups $\Delta^+$ and $\Delta^-$ are
    externally definable in $\Gamma$.

    As $\val(p) \in \Delta^+$, we know that $p$ has trivial valuation
    in $K \rightsquigarrow A$.  Therefore $A$ has characteristic 0.
    On the other hand, in $A \rightsquigarrow B$, $\val(p) \ne 0$,
    because $\val(p) \notin \Delta^-$.  Therefore $A \rightsquigarrow
    B$ is a mixed characteristic valuation.  It follows that
    \begin{align*}
      0 &= \characteristic(K) = \characteristic(A) \\
      p &= \characteristic(B) = \characteristic(k).
    \end{align*}
    Now $K \rightsquigarrow A$ is henselian by
    Conjecture~\ref{hens-conj}.  And $B \rightsquigarrow k$ is
    henselian by \cite{prdf}, Theorem~2.8.  Finally, $A
    \rightsquigarrow B$ is henselian because it is spherically
    complete\footnote{The value group of $A \rightsquigarrow B$ is a
      subgroup of $(\Rr,+,\le)$ so one only needs to consider
      countable chains.  Countable chains of balls have non-empty
      intersection because this held in $K \rightsquigarrow k$ by
      saturation.  For slightly more details, see \cite{arxiv-myself}
      \S 6.3, especially Remark~6.5 and the proof of Lemma~6.8.}
    \qedhere
  \end{itemize}
\end{proof}
So the henselianity conjecture reduces to the case of
equicharacteristic 0.

\section{Deformations and multiplicative infinitesimals} \label{first-of-multy}
In this section, $\Kk$ will be a sufficiently saturated unstable
dp-finite field.  We will (trivially) generalize the techniques of
\cite{prdf}, \S 6.3 to show that there is a good group of
``multiplicative $K$-infinitesimals'' for small models $K$.
\begin{definition}
  Let $K$ be a small model.  An affine symmetry 
  \begin{align*}
    f : \Kk & \to \Kk \\
    x & \mapsto a\cdot x + b
  \end{align*}
  is a \emph{$K$-deformation} if for every $K$-definable heavy set
  $X$, the intersection
  \begin{equation*}
    X \cap f^{-1}(X)
  \end{equation*}
  is heavy.
\end{definition}
We think of $K$-deformations as being the affine symmetries that are
``$K$-infinitesimally close'' to the identity map.
\begin{example}
  $f(x) = x + \varepsilon$ is a $K$-deformation if and only if
  $\varepsilon$ is a $K$-infinitesimal.
\end{example}
By Theorem~4.20.6-7 in \cite{prdf}, affine symmetries $x \mapsto
a\cdot x + b$ preserve heaviness.\footnote{This is the reason we are
  restricting to affine symmetries.  On the other hand,
  Theorem~\ref{new-mult-1}.\ref{nm22} below shows that \emph{all}
  definable bijections preserve heaviness.  Once this is known, the
  arguments of the present section could be repeated for general
  definable bijections.}
\begin{remark} \label{conj-remark}
  Let $K$ be a small model.  Let $f$ and $g$ be affine symmetries,
  with $f$ a $K$-deformation and $g$ being $K$-definable.  Then $g^{-1}
  \circ f \circ g$ is a $K$-deformation.
\end{remark}
\begin{proof}
  If $X$ is a $K$-definable heavy set, then $g(X)$ is also a
  $K$-definable heavy set, so
  \begin{equation*}
    g(X) \cap f^{-1}(g(X))
  \end{equation*}
  is heavy.  Therefore
  \begin{equation*}
    X \cap g^{-1}(f^{-1}(g(X)))
  \end{equation*}
  is heavy as well.
\end{proof}
We will show that the $K$-deformations are closed under composition
and inverses.  The proof is identical to \cite{prdf}, \S 6.3, but we
write out the details anyway.
\begin{definition}
  Let $K$ be a small model and $X \subseteq \Kk$ be $K$-definable.  An
  affine symmetry $f(x)$ is said to \emph{$K$-displace} $X$ if
  \begin{equation*}
    x \in X \cap x \in K \implies f(x) \notin X.
  \end{equation*}
\end{definition}
\begin{lemma}\label{heirs}
  Let $K \preceq K'$ be an inclusion of small models, and let $f(x) =
  a \cdot x + b$ and $f'(x) = a' \cdot x + b'$ be two affine symmetries of $\Kk$.
  Suppose that $\tp(a'b'/K')$ is an heir of $\tp(ab/K)$.
  \begin{enumerate}
  \item If $f$ is a $K$-deformation, then $f'$ is a $K'$-deformation.
  \item If $X \subseteq \Kk$ is $K$-definable and $K$-displaced by
    $f$, then $X$ is $K'$-displaced by $f'$.
  \end{enumerate}
\end{lemma}
\begin{proof}
  The assumptions imply that $a'b' \equiv_K ab$, and the statements
  about $ab$ are $K$-invariant, so we may assume $a'b' = ab$ and $f =
  f'$.
  \begin{enumerate}
  \item Suppose $f$ fails to be a $K'$-deformation.  Then there is a
    tuple $c \in \dcl(K')$ such that $\phi(\Kk;c)$ is heavy but
    \begin{equation*}
      \phi(\Kk;c) \cap f^{-1}(\phi(\Kk;c))
    \end{equation*}
    is light (not heavy).  These conditions on $c$ are
    $Kab$-definable.  Because $\tp(c/Kab)$ is finitely satisfiable in
    $K$, we can find such a $c$ in $K$.  Then $f$ fails to be a
    $K$-deformation because of the $K$-definable set $\phi(\Kk;c)$.
  \item Suppose $f$ fails to $K'$-displace $X$.  Then there exists $c
    \in \dcl(K')$ such that $c \in X$ and $f(c) \in X$.  These
    conditions on $c$ are $Kab$-definable, so we can find such a $c$
    in $K$.  Then $f$ fails to $K$-displace $X$. \qedhere
  \end{enumerate}
\end{proof}
\begin{lemma}
  \label{inf-heart}
  Let $K$ be a small model defining a critical coordinate
  configuration.  Let $f$ be a $K$-deformation, and $X
  \subseteq \Kk$ be a $K$-definable set that is $K$-displaced by
  $f$.  Then $X$ is light.
\end{lemma}
\begin{proof}
  Let $f(x) = a \cdot x + b$.  Build a sequence $a_0b_0, a_1b_1,
  \ldots$ and $K = K_0 \preceq K_1 \preceq K_2 \preceq \cdots$ so that
  for each $i$,
  \begin{itemize}
  \item $\tp(a_ib_i/K_i)$ is an heir of $\tp(ab/K)$.
  \item $K_{i+1} \ni a_ib_i$
  \end{itemize}
  Let $f_i(x) = a_i \cdot x + b_i$.  By Lemma~\ref{heirs},
  \begin{itemize}
  \item $f_i$ is a $K_i$-deformation.
  \item The set $X$ is $K_i$-displaced by $f_i$.
  \end{itemize}
  For $\alpha$ a string in $\{0,1\}^{< \omega}$, define $X_\alpha$
  recursively as follows:
  \begin{itemize}
  \item $X_{\{\}} = X$.
  \item If $\alpha$ has length $n$, then $X_{\alpha 0} = \{x \in
    X_\alpha ~|~ f_n(x) \notin X\}$.
  \item If $\alpha$ has length $n$, then $X_{\alpha 1} = \{x \in
    X_\alpha ~|~ f_n(x) \in X\}$.
  \end{itemize}
  For example
  \begin{align*}
    X_0 = \{x \in X ~|~ &f_0(x) \notin X\} \\
    X_1 = \{x \in X ~|~ &f_0(x)  \in X\} \\
    X_{011} = \{x \in X ~|~ &f_0(x) \notin X,\\ &f_1(x)
    \in X,\\& f_2(x) \in X\}.
  \end{align*}
  Note that some of the $X_\alpha$ must be empty, by NIP.  Note also
  that if $\alpha$ has length $n$, then $X_\alpha$ is $K_n$-definable.
  \begin{claim}
    If $X_\alpha$ is heavy, then $X_{\alpha 1}$ is heavy.
  \end{claim}
  \begin{claimproof}
    Let $\alpha$ have length $n$.  Then $X_\alpha$ is heavy and
    $K_n$-definable, and $f_n$ is a $K_n$-deformation.  Consequently
    $X_\alpha \cap f_n^{-1}(X_\alpha)$ is heavy.  But
    \begin{equation*}
      X_{\alpha 1} = X_\alpha \cap f_n^{-1}(X) \supseteq X_\alpha
      \cap f_n^{-1}(X_\alpha),
    \end{equation*}
    and so $X_{\alpha 1}$ must be heavy.
  \end{claimproof}
  \begin{claim}
    If $X_\alpha$ is heavy, then $X_{\alpha 0}$ is heavy.
  \end{claim}
  \begin{claimproof}
    Let $\alpha$ have length $n$.  The set $X_\alpha$ is
    $K_n$-definable.  Note that
    \begin{equation*}
      x \in X_\alpha(K_n) \implies x \in X(K_n) \implies f_n(x) \notin
      X(K_n)
    \end{equation*}
    because $X$ is $K_n$-displaced by $f_n$.  Therefore
    \begin{equation*}
      X_\alpha(K_n) \subseteq X_{\alpha 0}.
    \end{equation*}
    By \cite{prdf}, Lemma~4.22 it follows that $X_{\alpha 0}$ is
  \end{claimproof}
  If $X = X_{\{\}}$ is heavy, then the two claims imply that every
  $X_\alpha$ is heavy, hence non-empty, for every $\alpha$.  This
  contradicts NIP.
\end{proof}


\begin{lemma}
  \label{inf-subtr}
  Let $K$ be a model defining a critical coordinate configuration.
  Let $f_1, f_2$ be two $K$-deformations.  Then $f_2 \circ f_1^{-1}$
  is a $K$-deformation.
\end{lemma}
\begin{proof}
  Let $X$ be a $K$-definable heavy set; we will show that $X \cap
  f_1(f_2^{-1}(X))$ is heavy.  Note that $X$ is covered by the union
  of the following three sets:
  \begin{align*}
    D_0 &:= \{x \in X ~|~ f_1(x) \in X,~ f_2(x) \in X\}
    \\
    D_1 & := \{x \in \Kk ~|~ f_1(x) \notin X \} \\
    D_2 & := \{x \in \Kk ~|~ f_2(x) \notin X \}.
  \end{align*}
  By \cite{prdf}, Lemma~4.21, there is a $j \in \{0,1,2\}$ and a
  $K$-definable heavy set $X' \subseteq X$ such that $X'(K) \subseteq
  D_j$.  If $j > 0$, then
  \begin{equation*}
    x \in X'(K) \implies x \in D_j \implies f_j(x) \notin X
    \implies f_j(x) \notin X'.
  \end{equation*}
  In other words, $X'$ is $K$-displaced by $f_j$.  But $X'$ is heavy
  and $f_j$ is a $K$-deformation, so this would contradict
  Lemma~\ref{inf-heart}.  Therefore $j = 0$.  The fact that $X'(K)
  \subseteq D_0$ implies that $D_0$ is heavy, by \cite{prdf},
  Lemma~4.22.  By definition of $D_0$,
  \begin{align*}
    f_1(D_0) &\subseteq X \\
    f_2(D_0) &\subseteq X \\
    D_0 & \subseteq f_2^{-1}(X) \\
    f_1(D_0) &\subseteq f_1(f_2^{-1}(X)) \\
    f_1(D_0) & \subseteq X \cap f_1(f_2^{-1}(X)).
  \end{align*}
  Heaviness of $D_0$ then implies heaviness of $X \cap
  f_1(f_2^{-1}(X))$.
\end{proof}

\begin{theorem} \label{defo-group}
  If $K$ is any small model, the $K$-deformations form a subgroup of
  the $\Kk$-definable affine symmetries of $\Kk$.
\end{theorem}
\begin{proof}
  Let $f_1, f_2$ be two $K$-deformations.  We claim $f_1 \circ
  f_2^{-1}$ is a $K$-deformation.  Let $K' \succeq K$ be a small model
  defining $f_1, f_2$.  Let $K'' \succeq K$ be a small model defining
  a critical coordinate configuration.  Move $K''$ over $K$ so that
  $\tp(K''/K')$ is finitely satisfiable in $K$.  Let $a_i, b_i$ be
  such that $f_i(x) = a_i \cdot x + b_i$.  Then
  \begin{equation*}
    a_1a_2b_1b_2 \in \dcl(K')
  \end{equation*}
  and so $\tp(K''/Ka_ib_i)$ is finitely satisfiable in $K$ for $i =
  1,2$.  By Lemma~\ref{heirs}, $f_1$ and $f_2$ are $K''$-deformations.
  By Lemma~\ref{inf-subtr} $f_1 \circ f_2^{-1}$ is a $K''$-deformation.
  A fortiori, it is a $K$-deformation: if $X$ is any $K$-definable
  heavy set then $X$ is a $K''$-definable heavy set and so
  \begin{equation*}
    X \cap f_1(f_2^{-1}(X))
  \end{equation*}
  is heavy.
\end{proof}

\begin{definition}
  Let $K$ be a small model.  An element $\mu \in \Kk^\times$ is a
  \emph{multiplicative $K$-infinitesimal} if the map $x \mapsto \mu
  \cdot x$ is a $K$-deformation.

  Equivalently, $\mu$ is a multiplicative $K$-infinitesimal if for
  every $K$-definable heavy set $X$, the intersection
  \begin{equation*}
    X \cap (\mu \cdot X)
  \end{equation*}
  is heavy.
\end{definition}
\begin{theorem} \label{mult-1}
  Let $K$ be a small model.
  \begin{enumerate}
  \item \label{m11} Multiplicative $K$-infinitesimals form a subgroup
    $U_K \le \Kk^\times$.
  \item \label{m12} $U_K$ is type-definable over $K$.
  \item \label{interesting} If $\mu \in U_K$, then $\mu - 1$ is an
    (additive) $K$-infinitesimal.
  \item \label{min-heavy} Let $G$ be a subgroup of $\Kk^\times$,
    type-definable over $K$.  Suppose that for every $K$-definable set
    $D \supseteq G$, $D$ is heavy.  Then $U_K \le G$.
  \end{enumerate}
\end{theorem}
\begin{proof}~
  \begin{enumerate}
  \item Follows directly from Theorem~\ref{defo-group}.
  \item For any heavy $K$-definable set $X$, the set
    \begin{equation*}
      X' = \{ \mu \in \Kk^\times ~|~ X \cap (\mu \cdot X) \textrm{ is
        heavy}\}
    \end{equation*}
    is definable, because heaviness is definable in families
    (\cite{prdf}, Theorem~4.20.4).  As heaviness is
    $\Aut(\Kk/\emptyset)$-invariant, this set is $K$-definable.  The
    intersection of all such $X'$ is the group $U_K$ of multiplicative
    $K$-infinitesimals.
  \item Let $f(x) = x \cdot \mu$.  Let $g(x) = x + 1$.  Then $f$ is a
    $K$-deformation and $g$ is $K$-definable.  By
    Remark~\ref{conj-remark} and Theorem~\ref{defo-group}, the
    commutator $g^{-1} \circ f \circ g \circ f^{-1}$ is a
    $K$-deformation.  But this map is exactly
    \begin{equation*}
      x \mapsto \left(\frac{x}{\mu} + 1 \right) \cdot \mu - 1 = x +
      (\mu - 1).
    \end{equation*}
    Therefore, $\mu - 1$ is an (additive) $K$-infinitesimal.
  \item Note that
    \begin{equation*}
      G = \bigcap \{ D \cdot D^{-1} ~|~ D \textrm{ is $K$-definable, } D
      \supseteq G\},
    \end{equation*}
    where $D^{-1}$ denotes $\{x^{-1} ~|~ x \in D\}$.  Indeed, every
    set in the intersection contains $G$, and if $g$ is a point in the
    intersection, we can by compactness find $a, b \in G$ such that $g
    = a \cdot b^{-1}$.

    Now suppose that $\mu$ is a multiplicative $K$-infinitesimal.  By
    the assumption on $G$, all the sets $D$ are heavy.  By definition
    of mulitplicative infinitesimal, the sets $D \cap \mu \cdot D$ are
    heavy, hence \emph{non-empty}.  This means that $\mu \in D \cdot
    D^{-1}$ for all $D$, and so $\mu \in G$. \qedhere
  \end{enumerate}
\end{proof}

\section{Simultaneous coheir independence and dp-rank independence}
In a dp-finite structure, there are several senses in which two
elements $a, b$ can be ``independent'' over a small model $M$:
\begin{itemize}
\item coheir independence: $\tp(a/Mb)$ is finitely satisfiable in $M$.
\item dp-rank independence: $\dpr(ab/M) = \dpr(a/M) + \dpr(b/M)$.
\end{itemize}
It turns out that one can move $a, b$ over $M$ to arrange for both
conditions to hold simultaneously.  The proof is a bit confusing.

\begin{lemma}\label{point-of-confusion-0}
  Let $M$ be an $|A|^+$-saturated structure for some $A \subseteq M$.
  Suppose $\{\phi(x;b_{ij})\}_{i < r, j \in \Zz}$ is an ict-pattern of
  depth $r$ in some partial type $\Sigma(x)$ over $A$.  Then there
  exists $\{b'_{ij}\}_{i < r, j \in \Zz}$ and $\{a_\eta\}_{\eta : r
    \to \Zz}$ such that
  \begin{itemize}
  \item $\{\phi(x;b'_{ij})\}$ is an ict-pattern of depth $r$ in
    $\Sigma(x)$.
  \item The $a_\eta$ are witnesses:
    \begin{align*}
      M &\models \Sigma(a_\eta) \\
      M &\models \phi(a_\eta;b'_{ij}) \iff j = \eta(i).
    \end{align*}
  \item The array $\{b'_{ij}\}$ is mutually indiscernible over $A$.
  \item The type $\tp(a_\eta/A)$ is independent of $\eta$.
  \end{itemize}
\end{lemma}
\begin{proof}
  Well-known, except possibly the final point.  By the saturation
  assumption, we may replace $M$ with a bigger model.  Therefore we
  may assume $M$ is $|A|^+$-strongly homogeneous.  After extracting
  the mutually indiscernible array $\{b'_{ij}\}_{i < r, j \in \Zz}$,
  choose some witness $a_{\vec{0}}$ for the zero function $r \to \Zz$:
  \begin{align*}
    M & \models \Sigma(a_{\vec{0}}) \\ M & \models
    \phi(a_{\vec{0}};b'_{ij}) \iff j = 0.
  \end{align*}
  For any $\eta : r \to \Zz$, choose an automorphism $\sigma_\eta \in
  \Aut(M/A)$ sending $b'_{ij}$ to $b'_{i,j + \eta(i)}$.  Set $a_\eta =
  \sigma_\eta(a_{\vec{0}})$.  Then
  \begin{align*}
    M & \models \Sigma(a_\eta) \\
    M & \models \phi(a_\eta;b'_{i,j+\eta(i)}) \iff j = 0 \\
    M & \models \phi(a_\eta;b'_{i,j}) \iff j = \eta(i).
  \end{align*}
  Thus the $a_\eta$ are witnesses, as desired.  And $a_\eta \equiv_A
  a_{\vec{0}}$ for all $\eta$.
\end{proof}

\begin{lemma}\label{point-of-confusion-1}
  Let $M$ be a structure and $A \subseteq M$ be a subset.  Suppose $M$
  is $|A|^+$-saturated.  Suppose that \emph{in some reduct} $M_0$ of
  $M$,
  \begin{itemize}
  \item There is a partial type $\Sigma(x)$ over $A$.
  \item $\dpr(\Sigma(x)) \ge r$ for some finite $r$.
  \end{itemize}
  Then in $M$ there is an ict-pattern $\{\phi(x;b_{ij})\}_{i < r, j
    \in \Zz}$ and witnesses $\{a_\eta\}_{\eta : r \to \Zz}$ realizing
  $\Sigma(x)$, such that
  \begin{itemize}
  \item The formula $\phi(x;y)$ comes from the reduct language.
  \item The array $b_{ij}$ is mutually indiscernible over $A$,
    \emph{in the expansion}.
  \item The type of $a_\eta$ over $A$ \emph{in the expansion} is
    independent of $\eta$.
  \end{itemize}
\end{lemma}
\begin{proof}
  The reduct $M_0$ is also $|A|^+$-saturated, so we can find an
  ict-pattern $\{\phi(x;b_{ij})\}_{i < r, j \in \Zz}$ in $M_0$.  In
  particular, the formula $\phi(x;y)$ is from the reduct language.
  Now this ict-pattern continues to be an ict-pattern in the expansion
  $M$, so we can apply Lemma~\ref{point-of-confusion-0} there and
  obtain the desired indiscernibility in the expansion.
\end{proof}

\begin{lemma}\label{point-of-confusion-2}
  Let $\Mm$ be a monster model and $M \preceq M' \preceq \Mm$ be two
  small submodels.  There is a small submodel $N$ containing $M$ with
  the following properties:
  \begin{itemize}
  \item $\tp(N/M')$ is finitely satisfiable in $M$.
  \item The expansion of $N$ by all externally $M'$-definable sets is
    an $|M'|^+$-saturated structure.
  \end{itemize}
\end{lemma}
\begin{proof}
  Consider the theory of the pair $(M,M')$.  Take an
  $|M'|^+$-saturated elementary extension $(N,N') \succeq (M,M')$.
  Without loss of generality $N'$ is a submodel of $\Mm$.  Now the
  $\tp(N/M')$ is finitely satisfiable in $M$ for standard reasons.
  All the externally $M'$-definable sets in $N$ are definable in the
  structure $(N,N')$, and there are only a small number of them.
\end{proof}

Combining the previous two lemmas yields the following:
\begin{lemma}\label{point-of-confusion-3}
  Let $\Mm$ be a monster model and $M \preceq M' \preceq \Mm$ be two
  small submodels.  Let $\Sigma(x)$ be a partial type over $M$ with
  $\dpr(\Sigma(x)) \ge r$.  Then there is an ict-pattern
  $\{\phi(x;b_{ij})\}_{i < r, j \in \Zz}$ and witnesses
  $\{a_\eta\}_{\eta : r \to \Zz}$ such that
  \begin{itemize}
  \item The $b_{ij}$ are mutually indiscernible over $M'$.
  \item $\tp(a_\eta/M')$ is independent of $\eta$.
  \item $\tp(a_\eta/M')$ and $\tp(b_{ij}/M')$ are finitely satisfiable
    in $M$.
  \end{itemize}
\end{lemma}
\begin{proof}
  Take $N$ containing $M$ as in Lemma~\ref{point-of-confusion-2}.
  Then $\tp(N/M')$ is finitely satisfiable in $M$, and the expansion
  of $N$ by all externally $M'$-definable sets is an
  $|M'|^+$-saturated structure.  Then by
  Lemma~\ref{point-of-confusion-1}, there is an ict-pattern
  $\{\phi(x;b_{ij})\}_{i < r, j \in \Zz}$ and witnesses
  $\{a_\eta\}_{\eta : r \to \Zz}$ such that
  \begin{itemize}
  \item The formula $\phi(x;y)$ comes from the reduct language (the
    usual language).
  \item The array $b_{ij}$ is mutually indiscernible over $M$, in the
    expansion.  This implies that $b_{ij}$ is indiscernible over $M'$
    in the base language.
  \item The type of $a_\eta$ over $M$ \emph{in the expansion} is
    independent of $\eta$.  This implies that $\tp(a_\eta/M')$ is
    independent of $\eta$ in the base language.
  \end{itemize}
  Finally, $\tp(a_\eta/M')$ and $\tp(b_{ij}/M')$ are finitely
  satisfiable in $M$ because $\tp(N/M')$ is finitely satisfiable in
  $M$.
\end{proof}

\begin{proposition}\label{double-independence}
  Let $\Mm$ be a dp-finite monster model.  Let $M$ be a small
  substructure.  Given $a, b \in \Mm$, we can find $a' \equiv_M a$ and
  $b' \equiv_M b$ such that $\dpr(a'b'/M) = \dpr(a'/M) + \dpr(b'/M)$
  and $\tp(b'/Ma')$ is finitely satisfiable in $M$.
\end{proposition}
\begin{proof}
  Let $r = \dpr(a/M)$ and $s = \dpr(b/M)$.  Extract an ict-pattern
  $\{\varphi(x;c_{ij})\}_{i < r, j \in \Zz}$ in $\tp(a/M)$, mutually
  indiscernible over $M$.  Choose witnesses $\{a_\eta\}_{\eta : r \to
    \Zz}$ realizing $\tp(a/M)$, i.e.,
  \begin{equation*}
    j = \eta(i) \iff \models \varphi(a_\eta;c_{ij}).
  \end{equation*}
  Let $M'$ be a small model containing $M$, the $c_{ij}$, and the
  $a_\eta$.  By Lemma~\ref{point-of-confusion-3}, there is an
  ict-pattern $\{\psi(y;d_{ij})\}_{i < s, j \in \Zz}$ in the type
  $s(y)$, and witnesses $\{b_\eta\}_{\eta : s \to \Zz}$, such that
  \begin{itemize}
  \item The array $\{d_{ij}\}_{i < s, j \in \Zz}$ is mutually
    indiscernible over $M'$.
  \item $\tp(b_\eta/M')$ is independent of $\eta$, and finitely
    satisfiable in $M'$.
  \end{itemize}

  Now for any $\eta : r \to \Zz$ and $\eta' : s \to \Zz$, we have
  \begin{equation*}
    a_\eta b_{\eta'} \equiv_M a_{\vec{0}} b_{\eta'} \equiv_M
    a_{\vec{0}} b_{\vec{0}}.
  \end{equation*}
  The first $\equiv$ holds because
  \begin{equation*}
    \tp(a_\eta/M) = \tp(a/M) = \tp(a_{\vec{0}}/M)
  \end{equation*}
  and $\tp(b_{\eta'}/Ma_\eta a_{\vec{0}})$ is finitely satisfiable in
  $M$.  The second $\equiv$ holds because $Ma_{\vec{0}} \subseteq M'$ and
  \begin{equation*}
    \tp(b_{\eta'}/M') = \tp(b_{\vec{0}}/M').
  \end{equation*}
  Therefore $\tp(a_\eta b_{\eta'}/M)$ is independent of $\eta$ and
  $\eta'$.  Call this type $q(x,y)$.  There is an ict-pattern of depth
  $r + s$ in the complete type $q(x,y)$ given by the following array
  of formulas:
  \begin{align*}
    &\varphi(x,c_{1,1}), ~\varphi(x,c_{1,2}),~ \ldots \\
    & \ldots \\
    & \varphi(x,c_{r,1}),~\varphi(x,c_{r,2}),~ \ldots \\
    &\psi(y,d_{1,1}), ~\psi(y,d_{1,2}),~ \ldots \\
    & \ldots \\
    & \psi(y,d_{s,1}),~\psi(y,d_{s,2}),~ \ldots
  \end{align*}
  Indeed, the $a_\eta b_{\eta'}$ witness that this is an ict-pattern.
  Thus $q(x,y)$ has dp-rank at least $r + s$.  Take $a'b' \models q |
  M$.  Then $a' \equiv_M a$ and $b' \equiv_M b$, and $\dpr(a'b'/M) = r
  + s$.
\end{proof}

\section{Riddles answered} \label{last-of-multy}
In this section, we show that heavy sets are exactly sets of full
rank, the additive infinitesimals agree with the multiplicative
infinitesimals, Sinclair's ``Johnson topology'' exists and agrees with
the canonical topology, and the canonical topology is a field
topology.  As in \S \ref{first-of-multy}, $\Kk$ will be a sufficiently
saturated unstable dp-finite field.

We will make use of the following facts which were implicit in \cite{prdf}.
\begin{theorem} \label{proto-shrink}
  Let $Q_1, \ldots, Q_n$ be quasi-minimal sets and let $P \subseteq
  Q_1 \times \cdots \times Q_n$ have full rank
  \begin{equation*}
    \dpr(P) = \dpr(Q_1 \times \cdots \times Q_n).
  \end{equation*}
  Then there are smaller quasi-minimal sets $Q_i' \subseteq Q_i$ such
  that
  \begin{align*}
    \dpr((Q_1' \times \cdots \times Q_n') \cap P) &= \dpr(Q_1 \times
    \cdots \times Q_n) \\
    \dpr((Q_1' \times \cdots \times Q_n') \setminus P) &< \dpr(Q_1
    \times \cdots \times Q_n).
  \end{align*}
\end{theorem}
\begin{proof}
  If $n = 1$, take $Q_1' = P$.  Then $P$ is infinite by \cite{prdf},
  Remark~3.2, and so $\dpr(Q_1') = \dpr(Q_1)$ by definition of
  quasi-minimality.

  Assume therefore that $n > 1$.  By \cite{prdf}, Theorem~3.23, the
  set $P$ is broad in $Q_1 \times \cdots \times Q_n$.  By \cite{prdf},
  Theorem~3.10, there exist infinite definable subsets $Q_i' \subseteq
  Q_i$ such that the set
  \begin{equation*}
    H = (Q_1' \times \cdots \times Q_n') \setminus P
  \end{equation*}
  is a ``hyperplane,'' in the sense that for every $b \in Q_n'$, the
  set
  \begin{equation*}
    \{(a_1,\ldots,a_{n-1}) \in Q_1' \times \cdots \times Q_{n-1}' ~|~
    (a_1,\ldots,a_{n-1},b) \notin P\}
  \end{equation*}
  is narrow in $Q_1' \times \cdots \times Q_{n-1}'$.  By the
  contrapositive of \cite{prdf}, Lemma~3.8.1, $H$ is narrow in $Q_1'
  \times \cdots \times Q_n'$.  By \cite{prdf}, Theorem~3.23,
  \begin{equation*}
    \dpr((Q_1' \times \cdots \times Q_n') \setminus P) = \dpr(H) <
    \dpr(Q_1') + \cdots + \dpr(Q_n') = \dpr(Q_1' \times \cdots \times
    Q_n').
  \end{equation*}
  Because of how dp-rank behaves in unions, it follows that
  \begin{equation*}
    \dpr((Q_1' \times \cdots \times Q_n') \cap P) = \dpr(Q_1' \times
    \cdots \times Q_n').
  \end{equation*}
  Meanwhile, by definition of quasi-minimality, $\dpr(Q_i') =
  \dpr(Q_i)$, and so
  \begin{align*}
    \dpr(Q_1' \times \cdots \times Q_n') & = \dpr(Q_1') + \cdots +
    \dpr(Q_n') \\ & = \dpr(Q_1) + \cdots + \dpr(Q_n) \\ & = \dpr(Q_1 \times
    \cdots \times Q_n). \qedhere
  \end{align*}
\end{proof}
Recall the notion of coordinate configuration, critical rank, critical
sets, and heavy sets from \cite{prdf} (Definitions~4.1, 4.7, and
4.19).  We will use $\rho$ to denote the critical rank.

If $W$ is critical, then $\dpr(W) = \rho$ by definition of critical
set and by \cite{prdf}, Remark~4.2.  Note that any critical set $W$ is
heavy, as it is trivially $W$-heavy (\cite{prdf}, Definition~4.16).

\begin{corollary}\label{shrinking}
  Let $(Q_1,\ldots,Q_n,P)$ be a coordinate configuration of rank $r$.
  Then there is a coordinate configuration $(Q_1',\ldots,Q_n',P')$
  such that
  \begin{enumerate}
  \item $(Q_1',\ldots,Q_n',P')$ also has rank $r$.
  \item Each $Q_i'$ is a subset of $Q_i$.
  \item $P' = P \cap (Q_1' \times \cdots \times Q_n')$.
  \item The complement $(Q_1' \times \cdots \times Q_n') \setminus P'$
    is narrow, so
    \begin{equation*}
      \dpr((Q_1' \times \cdots \times Q_n') \setminus P') < r
    \end{equation*}
  \end{enumerate}
\end{corollary}

\subsection{Near interior}
In the classification of dp-minimal fields, an important step was that
infinite definable sets have non-empty interior (\cite{arxiv-myself}
Proposition 4.12).  In terms of infinitesimals, this says:
\begin{quote}
  If $X$ is an infinite $K$-definable set, there is $a \in X(K)$ such
  that $a + I_K \subseteq X$.
\end{quote}
This was used to show that $1 + I_K$ is the group of multiplicative
infinitesimals.\footnote{See Claim 4.13 in \cite{arxiv-myself} for the
  implication $1 + I_K \subseteq U_K$.  The other inclusion was
  omitted from \cite{arxiv-myself}, but was proven in earlier drafts
  using the method of Theorem~\ref{mult-1}.\ref{interesting} above.}

The analogue here would replace ``infinite'' with ``heavy.''  In this
section, we prove a weaker version, Proposition~\ref{where-we-are-going},
which provides an $a \in K$ such that
\begin{equation*}
  \left(\epsilon \in I_K \text{ and } \dpr(\epsilon/K) \ge \rho
  \right) \implies a + \epsilon \in X.
\end{equation*}
So instead of ensuring that \emph{all} points near $a$ are in $X$, the
lemma ensures that \emph{most} points near $a$ are in $X$.

\begin{lemma} \label{inject}
  Let $X \subseteq \Kk$ be a heavy definable set.  Then there is a
  critical set $W$ and a $\delta \in \Kk$ such that
  \begin{equation*}
    \dpr(\{(x,y) \in W \times W ~|~ x - y \notin X + \delta\}) < \dpr(W
    \times W) = 2 \rho.
  \end{equation*}
\end{lemma}
The idea here is that $x - y \in X + \delta$ for ``almost all'' $(x,y)
\in W \times W$.
\begin{proof}
  Let $(Q_1,\ldots,Q_n,P)$ be a critical coordinate configuration.
  Shrinking the $Q_i$ (Corollary~\ref{shrinking}), we may assume that
  \begin{align*}
    \dpr((Q_1 \times \cdots \times Q_n) \setminus P) &< \rho \\
    \dpr((Q_1 \times \cdots \times Q_n) \cap P) & = \rho.
  \end{align*}
  By \cite{prdf}, Corollary~4.15, we may find a $\delta \in \Kk$ such
  that
  \begin{equation*}
    \{ (\vec{x},\vec{y}) \in \prod_i Q_i \times \prod_i Q_i ~|~ \sum
    \vec{x} - \sum \vec{y} \in X + \delta\}
  \end{equation*}
  is broad as a subset of $\prod_i Q_i \times \prod_i Q_i$.  By
  Theorem~\ref{proto-shrink} we may find infinite definable $Q_i',
  Q_i'' \subseteq Q_i$ such that
  \begin{equation*}
    \{ (\vec{x},\vec{y}) \in \prod_i Q'_i \times \prod_i Q''_i ~|~
    \sum \vec{x} - \sum \vec{y} \notin X + \delta\}
  \end{equation*}
  is narrow as a subset of $\prod_i Q'_i \times \prod_i Q''_i$.  In
  particular,
  \begin{equation*}
    \dpr( \{ (\vec{x},\vec{y}) \in \prod_i Q'_i \times \prod_i Q''_i
    ~|~ \sum \vec{x} - \sum \vec{y} \notin X + \delta\}) < 2 \rho.
  \end{equation*}
  Let
  \begin{align*}
    P' & = P \cap (Q_1' \times \cdots \times Q_n') \\
    P'' & = P \cap (Q_1'' \times \cdots \times Q_n'').
  \end{align*}
  By choice of the $Q'_i$ and $Q''_i$, the sets $P'$ and $P''$ are broad, so
  $(Q_1',\ldots,Q_n',P')$ and $(Q_1'',\ldots,Q_n'',P'')$ are both
  critical coordinate configurations.  Let $Y'$ and $Y''$ be the
  respective targets.  Note that
  \begin{equation*}
    \{ (\vec{x},\vec{y}) \in P' \times P'' ~|~ \sum \vec{x} - \sum
    \vec{y} \notin X + \delta\} \subseteq \{ (\vec{x},\vec{y}) \in
    \prod_i Q'_i \times \prod_i Q''_i ~|~ \sum \vec{x} - \sum \vec{y}
    \notin X + \delta\}
  \end{equation*}
  and so
  \begin{equation*}
    \dpr( \{ (\vec{x},\vec{y}) \in P' \times P'' ~|~ \sum \vec{x} -
    \sum \vec{y} \notin X + \delta\}) < 2 \rho.
  \end{equation*}
  As $Y'$ and $Y''$ are the images of $P'$ and $P''$ under the maps
  $\vec{x} \mapsto \sum \vec{x}$, it follows that
  \begin{equation*}
    \dpr( \{(x,y) \in Y' \times Y'' ~|~ x - y \notin X + \delta\}) < 2
    \rho.
  \end{equation*}
  As $Y'$ and $Y''$ are the targets of critical coordinate
  configurations, each is a critical set, hence heavy.  By
  \cite{prdf}, Theorem~4.20.8, there is $\tau \in \Kk$ such that
  \begin{equation*}
    W := Y' \cap (Y'' + \tau)
  \end{equation*}
  is heavy.  Then
  \begin{equation*}
    \dpr( \{(x,y) \in W \times W ~|~ x - y + \tau \notin X + \delta
    \}) < 2 \rho.
  \end{equation*}
  Note that $W \subseteq Y'$, so
  \begin{equation*}
    \dpr(W) \le \dpr(Y') = \rho.
  \end{equation*}
  On the other hand, $W$ is heavy, so $\dpr(W) \ge \rho$.  Therefore
  $\dpr(W) = \rho$, so by \cite{prdf}, Remark~4.8 the set $W$ is
  critical.  And
  \begin{equation*}
    \dpr( \{(x,y) \in W \times W ~|~ x - y \notin X + \delta' \}) <
    \dpr(W \times W).
  \end{equation*}
  for $\delta' = \delta - \tau$.
\end{proof}
We need a slightly stronger version of Lemma~\ref{inject}, controlling
the field of definition of $W$ and $\delta$:
\begin{lemma} \label{inject-2}
  Let $K$ be a small model defining a critical coordinate
  configuration.  Let $X$ be heavy and $K$-definable.  Then there is a
  $K$-definable critical set $W$ and a $\delta \in K$ such that
  \begin{equation*}
    \dpr(\{(x,y) \in W \times W ~|~ x - y \notin X + \delta\}) < \dpr(W
    \times W) = 2 \rho.
  \end{equation*}
\end{lemma}
\begin{proof}
  Let $(Q_1,\ldots,Q_n,P)$ be some $K$-definable coordinate
  configuration with target $Y$.  Note that $\dpr(Y \times Y) = 2
  \rho$.
  \begin{claim}\label{square-definability}
    If $\{D_b\}$ is a $K$-definable family of subsets of $Y \times Y$,
    then 
    \begin{equation}
      \{b ~|~ \dpr(D_b) = 2 \rho\} \label{in-question}
    \end{equation}
    is $K$-definable.
  \end{claim}
  \begin{claimproof}
    Note that there is a surjection with finite fibers
    \begin{align*}
      P \times P & \twoheadrightarrow Y \times Y \\
      (\vec{x},\vec{y}) & \mapsto \left( \sum \vec{x} , \sum \vec{y} \right).
    \end{align*}
    Call this surjection $s$.  Then $D_b$ and $s^{-1}(D_b)$ have the
    same dp-rank, and we reduce to showing that the following set is $K$-definable:
    \begin{equation*}
      \{ b ~|~ \dpr(s^{-1}(D_b)) = \dpr(\prod_i Q_i \times \prod_i Q_i)\}
    \end{equation*}
    This follows by \cite{prdf}, Corollary~3.24.
  \end{claimproof}
  Now by Lemma~\ref{inject} we can find a $\delta_0$ and critical set
  $W$ such that
  \begin{equation*}
    \{(x,y) \in W \times W ~|~ x - y \notin X + \delta_0\}
  \end{equation*}
  has dp-rank less than $2 \rho$.  By \cite{prdf} (Remark~4.9 and
  Proposition~4.18), we may translate $W$ and arrange for $\dpr(W \cap
  Y) = \rho$.  By \cite{prdf}, Remark~4.8 the intersection $W' := W
  \cap Y$ is itself a critical set.  Note that
  \begin{equation*}
    \{(x,y) \in W' \times W' ~|~ x - y \notin X + \delta_0\}
  \end{equation*}
  has dp-rank less than $2 \rho$.

  Write $W'$ as $\varphi(\Kk;b_0)$.  The following condition on $b,
  \delta$ is $K$-definable by Claim~\ref{square-definability} and
  \cite{prdf}, Proposition~4.3.
  \begin{itemize}
  \item $\varphi(\Kk;b) \subseteq Y$, and
  \item $\dpr(\varphi(\Kk;b)) = \rho$, and
  \item The set
    \begin{equation*}
      \{(x,y) \in \varphi(\Kk;b) \times \varphi(\Kk;b) ~|~ x - y \notin X + \delta\}
    \end{equation*}
    has dp-rank less than $2 \rho$.
  \end{itemize}
  Therefore we can find $b, \delta \in K$ satisfying these conditions.
  Then $W'' := \varphi(\Kk;b)$ is a critical set, by Remark~4.8 in
  \cite{prdf} and the second requirement on $b, \delta$.  And by the
  third requirement,
  \begin{equation*}
    \{(x,y) \in W'' \times W'' ~|~ x - y \notin X + \delta\}
  \end{equation*}
  has dp-rank less than $2 \rho$.
\end{proof}

Let $I_K$ denote the group of (additive) $K$-infinitesimals.
\begin{proposition}\label{where-we-are-going}
  Let $K$ be a small model defining a critical coordinate
  configuration.  Let $X \subseteq \Kk$ be a heavy $K$-definable set.
  Then there is $a \in K$ such that for any $\epsilon \in I_K$,
  \begin{equation*}
    \dpr(\epsilon/K) \ge \rho \implies a + \epsilon \in X.
  \end{equation*}
\end{proposition}
\begin{proof}
  By Lemma~\ref{inject-2}, find a $K$-definable critical set $W$ and
  an $a \in K$ such that
  \begin{equation*}
    \{(x,y) \in W \times W ~|~ x - y \notin X - a\}
  \end{equation*}
  has rank less than $2 \rho$.  Let $\epsilon \in I_K$ have rank at
  least $\rho$.  Note that $W(K)$ is covered by the following
  $\Kk$-definable sets:
  \begin{align*}
    D_0 & := \{x \in W ~|~ x + \epsilon \in W \} \\
    D_1 & := \{x \in W ~|~ x + \epsilon \notin W\}
  \end{align*}
  As $W$ is $K$-definable and heavy, by Lemma~4.21 in \cite{prdf}
  there must be a $K$-definable heavy set $Y \subseteq W$ such that
  \begin{equation*}
    Y(K) \subseteq D_i
  \end{equation*}
  for $i = 0$ or $i = 1$.  In fact, $i = 1$ cannot happen: otherwise
  \begin{equation*}
    x \in Y(K) \implies x \in D_1 \implies x + \epsilon \notin W
    \implies x + \epsilon \notin Y.
  \end{equation*}
  Then $Y$ is $K$-displaced by $\epsilon$.  By
  Lemma~\ref{inf-heart}, it follows that $Y$ is light, a
  contradiction.

  So $i = 0$, and $Y(K) \subseteq D_0$.  In particular,
  \begin{equation}\label{callback}
    (x \in K \wedge x \in Y) \implies x + \epsilon \in W.
  \end{equation}
  Note that $\dpr(Y) = \rho$ because $Y$ is heavy and $Y \subseteq W$.
  Take $b \in Y$ such that $\dpr(b/K) = \dpr(Y) = \rho$.  By
  Proposition~\ref{double-independence}, we may move $b$ over $K$ and
  arrange for
  \begin{equation*}
    \dpr(\epsilon b / K) = \dpr(\epsilon / K) + \dpr(b / K) =
    \dpr(\epsilon / K) + \rho \ge 2 \rho
  \end{equation*}
  and for $\tp(b / K \epsilon)$ to be finitely satisfiable in $K$.  If
  $b + \epsilon \notin W$, then by finite satisfiability there is $b'
  \in K$ such that $b' + \epsilon \notin W$ and $b' \in Y$,
  contradicting (\ref{callback}).  Therefore $b + \epsilon \in W$.

  Now
  \begin{equation*}
    \dpr(b,\epsilon/K) = \dpr(b,b + \epsilon/K) \ge 2 \rho,
  \end{equation*}
  so
  \begin{equation*}
    (b + \epsilon, b) \notin \{(x,y) \in W \times W ~|~ x - y \notin X
    - a\}.
  \end{equation*}
  But $(b + \epsilon, b) \in W \times W$, so
  \begin{equation*}
    (b + \epsilon) - b \in X - a.
  \end{equation*}
  In other words, $a + \epsilon \in X$.
\end{proof}

\subsection{The critical rank}
We can now show that the critical rank $\rho$ is as large as possible,
and that heavy sets are merely sets of full rank.

\begin{proposition}\label{old-mult-1}
  Let $K$ be a small model defining a critical coordinate configuration, and let
  $I_K$ be the group of $K$-infinitesimals.  Then $\dpr(I_K) \le \rho$.
\end{proposition}
\begin{proof}
  Suppose for the sake of contradiction that $\dpr(I_K) = \rho'
  > \rho$.  Choose some $K$-infinitesimal $\epsilon$ such that
  $\dpr(\epsilon/K) = \rho'$.  Let $Y$ be a critical $K$-definable
  set.  Then $Y$ is heavy and $\dpr(Y) = \rho$.  By
  Proposition~\ref{where-we-are-going}, there is some $a \in K$ such that
  $a + \epsilon \in Y$.  Then
  \begin{equation*}
    \rho' = \dpr(\epsilon / K) = \dpr( a + \epsilon / K) \le \dpr(Y)
    = \rho,
  \end{equation*}
  a contradiction.
\end{proof}
\begin{lemma}\label{lemma-surprise}
  Let $K$ be a small model.  Then $\dpr(I_K) = \dpr(\Kk)$.
\end{lemma}
\begin{proof}
  Let $n = \dpr(\Kk)$.  Choose an ict-pattern
  $\{\varphi(x;b_{ij})\}_{i < n, j < \omega}$ of depth $n$ in $\Kk$.
  For each $\eta : n \to \omega$ let $a_\eta \in \Kk$ be a witness, so
  that
  \begin{equation*}
    \models \varphi(a_\eta,b_{ij}) \iff j = \eta(i).
  \end{equation*}
  Let $K'$ be a small model containing $K$ and the $a_\eta$.  Let
  $\epsilon$ be a non-zero $K'$ infinitesimal.  Then $a_\eta \in
  \epsilon^{-1} \cdot I_{K'}$ for every $\eta$, by \cite{prdf},
  Remark~6.9.3.  So there is an ict-pattern of depth $n$ in
  $\epsilon^{-1} \cdot I_{K'}$.  Therefore
  \begin{equation*}
    \dpr(I_K) \ge \dpr(I_{K'}) = \dpr(\epsilon^{-1} \cdot I_{K'}) \ge n,
  \end{equation*}
  where the first inequality holds because $I_{K'} \subseteq I_K$.
\end{proof}

\begin{theorem}\label{new-mult-1}
  ~
  \begin{enumerate}
  \item The critical rank $\rho$ is exactly $\dpr(\Kk)$.
  \item \label{nm22} A definable set $X \subseteq \Kk$ is heavy if and only if
    $\dpr(X) = \dpr(\Kk)$.  In particular, this condition is
    definable in families.
  \end{enumerate}
\end{theorem}
\begin{proof}~
  \begin{enumerate}
  \item Proposition~\ref{old-mult-1} and Lemma~\ref{lemma-surprise},
    recalling that the critical rank is at most $\dpr(\Kk)$ by
    definition.
  \item If $\dpr(X) = \dpr(\Kk)$, then $X$ is heavy by \cite{prdf},
    Lemma~7.1.  Conversely, if $X$ is heavy, then $X$ contains a
    critical set, so $\dpr(X) \ge \rho = \dpr(\Kk)$.  Definability
    then follows from definability of heaviness (\cite{prdf},
    Theorem~4.20.4). \qedhere
  \end{enumerate} 
\end{proof}

\begin{corollary}\label{alternative-bases}
  Let $K$ be an unstable field of dp-rank $n$.  As $X$ ranges over the
  rank-$n$ definable subsets of $K$, each of the following ranges over
  a neighborhood basis of 0 in the canonical topology:
  \begin{align*}
    X \mininf X = \{\delta \in K ~&|~ X \cap (X + \delta) \textrm{ is heavy}\} \\
    X \ominus X := \{\delta \in K ~&|~ X \cap (X + \delta) \textrm{ is infinite}\} \\
    X - X = \{\delta \in K ~&|~ X \cap (X + \delta) \textrm{ is non-empty}\}
  \end{align*}
  In particular, Sinclair's ``Johnson topology'' exists and
  agrees with the canonical topology (see \cite{sinclair} \S 3.2).
\end{corollary}
\begin{proof}
  The family $\{X \mininf X\}$ is a neighborhood basis of 0 by
  definition of the canonical topology.  As $X \mininf X \subseteq X
  \ominus X \subseteq X - X$, the families $\{X \ominus X\}$ and $\{X
  - X\}$ are certainly families of neighborhoods.  To see that they
  are neighborhood bases, we only need to show that they are cofinal
  with the family $\{X \mininf X\}$, which is the following claim:
  \begin{claim}
    For any heavy set $X$ there is a heavy set $Y$ such that
    \begin{equation*}
      Y \ominus Y \subseteq Y - Y \subseteq X \mininf X.
    \end{equation*}
  \end{claim}
  \begin{claimproof}
    Choose a monster model $\Kk \succeq K$ and note that the
    following type is inconsistent:
    \begin{align*}
      x &\in I_K \\
      y &\in I_K \\
      x - y &\notin X \mininf X
    \end{align*}
    because $I_K - I_K = I_K \subseteq X \mininf X$.  Therefore, by
    compactness, there is a $K$-definable set $Y \supseteq I_K$ such
    that
    \begin{equation*}
      (x \in Y \textrm{ and } y \in Y) \implies x - y \in X \mininf X,
    \end{equation*}
    meaning that $Y - Y \subseteq X \mininf X$.  The set $Y$ contains
    a $K$-definable basic neighborhood, and is therefore heavy by
    \cite{prdf}, Proposition~6.5.1.
  \end{claimproof}
  Now by Theorem~\ref{new-mult-1}, the heavy sets are exactly the rank
  $n$ sets, so the family $\{X \ominus X\}$ is the family of basic
  neighborhoods of 0 in Sinclair's Johnson topology.
\end{proof}

\subsection{Infinitesimals and multiplication}
We follow a strategy similar to (\cite{arxiv-myself}, Claim 4.14), using the
existence of (near) interior to relate multiplicative and additive
infinitesimals.

\begin{proposition}\label{old-mult-2}
  Let $K$ be a small model defining a critical coordinate configuration, and let
  $I_K$ be the group of $K$-infinitesimals.  Then the group $U_K$ of
  multiplicative $K$-infinitesimals is exactly $1 + I_K$.
\end{proposition}
\begin{proof}
  The inclusion $U_K \subseteq 1 + I_K$ is
  Theorem~\ref{mult-1}.\ref{interesting}.  We must show the converse.
  
  By Lemma~\ref{lemma-surprise} and Proposition~\ref{old-mult-1},
  $\dpr(I_K) = \rho = \dpr(\Kk)$.  Fix a $K$-infinitesimal $\epsilon_0 \in I_K$ and a
  $K$-definable heavy set $X \subseteq \Kk$.  We will show that the
  intersection of $X$ and $(1 + \epsilon_0)^{-1} X$ is heavy.

  By Proposition~\ref{where-we-are-going}, find $a \in K$ such that for
  any $\epsilon \in I_K$,
  \begin{equation*}
    \dpr(\epsilon/K) = \rho \implies a + \epsilon \in X.
  \end{equation*}
  \begin{claim} \label{eps-2}
    If $\epsilon_1$ is a $K$-infinitesimal such that
    $\dpr(\epsilon_1 / K \epsilon_0) = \rho$, then
    \begin{equation*}
      (a + \epsilon_1) \in X \cap (1 + \epsilon_0)^{-1} X.
    \end{equation*}
  \end{claim}
  \begin{claimproof}
    The element
    \begin{equation*}
      \epsilon_2 := a \cdot \epsilon_0 + \epsilon_1 + \epsilon_1
      \cdot \epsilon_0.
    \end{equation*}
    is a $K$-infinitesimal by \cite{prdf} (Remark~6.9.3, Theorem~6.17,
    and Corollary~10.5).  Note that
    \begin{equation*}
      (a + \epsilon_1) (1 + \epsilon_0) = a + a \cdot \epsilon_0 +
      \epsilon_1 + \epsilon_1 \cdot \epsilon_0 = a + \epsilon_2.
    \end{equation*}
    Thus $\epsilon_1$ and $\epsilon_2$ are inter-definable over $K
    \epsilon_0$, and so
    \begin{equation*}
      \dpr(\epsilon_2 / K \epsilon_0) = \dpr(\epsilon_1 / K
      \epsilon_0) = \rho.
    \end{equation*}
    As $\dpr(I_K) = \rho$, this implies
    \begin{equation*}
      \dpr(\epsilon_2 / K) = \dpr(\epsilon_1 / K) = \rho.
    \end{equation*}
    By choice of $a$, it follows that
    \begin{align*}
      (a + \epsilon_1) & \in X \\
      (a + \epsilon_1)(1 + \epsilon_0) = (a + \epsilon_2) & \in X
    \end{align*}
    Therefore $(a + \epsilon_1) \in X \cap (1 + \epsilon_0)^{-1} X$.
  \end{claimproof}
  Now let $S$ be the type-definable set
  \begin{equation*}
    S = \{\epsilon \in I_K ~|~ a + \epsilon \notin X \cap (1 +
    \epsilon_0)^{-1} X\}.
  \end{equation*}
  Then $\dpr(S) \le \dpr(I_K) = \rho$.  We claim the inequality is
  strict.  Otherwise, we can find $\epsilon \in S$ with
  $\dpr(\epsilon / K \epsilon_0) = \rho$, because $S$ is
  type-definable over $K \epsilon_0$.  Then by Claim~\ref{eps-2},
  $\epsilon \notin S$, a contradiction.  Thus $\dpr(S) < \rho$, and
  so $\dpr(I_K \setminus S) = \rho$.  But the set
  \begin{equation*}
    I_K \setminus S = \{ \epsilon \in I_K ~|~ a + \epsilon \in X
    \cap ( 1 + \epsilon_0)^{-1} X\}
  \end{equation*}
  has a tautological embedding into $X \cap (1 + \epsilon_0)^{-1} X$, and so 
  \begin{equation*}
    \dpr(X \cap (1 + \epsilon_0)^{-1} X) \ge \rho = \dpr(\Kk).
  \end{equation*}
  Thus $X \cap (1 + \epsilon_0)^{-1}X$ is heavy.  We have shown that
  $1 + \epsilon_0$ is a multiplicative infinitesimal, and therefore
  that $1 + I_K \subseteq U_K$.
\end{proof}

\begin{theorem}\label{new-mult-2}
  Let $K \preceq \Kk$ be any small model.  If $I_K$ denotes the additive
  infinitesimals and $U_K$ denotes the multiplicative infinitesimals,
  then $1 + I_K = U_K$.
\end{theorem}
\begin{proof}
  By Proposition~\ref{old-mult-2}, it suffices to prove the following
  \begin{quote}
    If $K \preceq K'$ and if $1 + I_{K'} = U_{K'}$, then $1 + I_K =
    U_K$.
  \end{quote}
  Suppose $1 + I_{K'} =U_{K'}$.  By
  Theorem~\ref{mult-1}.\ref{interesting} the inclusion $U_K \subseteq
  1 + I_K$ always holds.  We must show $1 + I_K \subseteq U_K$.  Let
  $X$ be a heavy $K$-definable set.  Consider the set
  \begin{equation*}
    Y = \{\mu \in \Kk ~|~ X \cap (\mu \cdot X) \textrm{ is heavy}\}.
  \end{equation*}
  It suffices to show that $1 + I_K \subseteq Y$.  As $X$ is
  $K'$-definable, we have $1 + I_{K'} = U_{K'} \subseteq Y$.
  Therefore, there is a $K'$-definable heavy set $N \supseteq I_{K'}$
  such that
  \begin{equation*}
    1 + N \subseteq Y.
  \end{equation*}
  By Proposition~6.5 in \cite{prdf}, we may take $N$ to be a basic
  neighborhood (see \cite{prdf}, Definition~6.3), and so
  \begin{equation*}
    1 + (Z \mininf Z) \subseteq Y,
  \end{equation*}
  where $N = Z \mininf Z$ is as in \cite{prdf}, Definition~6.1.

  As heaviness is definable in families, we can pull the
  parameters defining $Z$ down into $K$, finding a $K$-definable heavy
  set $Z'$ such that
  \begin{equation*}
    1 + (Z' \mininf Z') \subseteq Y,
  \end{equation*}
  Then $1 + I_K \subseteq 1 + (Z' \mininf Z') \subseteq Y$.
\end{proof}

\begin{corollary}\label{field-topology-cor}
  The canonical topology of \cite{prdf} (Remark~6.18 and
  Corollary~10.5) is a \emph{field} topology---division is continuous.
\end{corollary}

As another consequence, we can slightly simplify the description of
the multiplicative infinitesimals:
\begin{proposition}\label{prop:psh-whatever}
  Let $K$ be a small model.  The group $U_K$ of multiplicative
  $K$-infinitesimals is exactly
  \begin{equation*}
    \bigcap \{X \cdot X^{-1} : X \subseteq
    \Kk^\times \text{ heavy and $K$-definable}\}
  \end{equation*}
\end{proposition}
\begin{proof}
  Let $X \div X$ and $X \div_\infty X$ denote the sets
  \begin{align*}
    X \div X &= \{\mu \in \Kk^\times : X \cap (\mu \cdot X) \text{ is
      non-empty}\} = X \cdot X^{-1} \\ X \div_\infty X &= \{\mu \in
    \Kk^\times : X \cap (\mu \cdot X) \text{ is heavy}\}
  \end{align*}
  As in the proof of Theorem~\ref{mult-1}.\ref{m12}, the set $X
  \div_\infty X$ is definable and
  \begin{equation*}
    U_K = \bigcap \{X \div_\infty X : X \subseteq \Kk^\times \text{
      heavy and $K$-definable}\}.
  \end{equation*}
  But $X \div_\infty X \subseteq X \div X$, so
  \begin{align*}
    & \bigcap \{X \div_\infty X : X \subseteq \Kk^\times \text{ heavy
      and $K$-definable}\} \\ \subseteq & \bigcap \{X \div X : X
    \subseteq \Kk^\times \text{ heavy and $K$-definable}\}
  \end{align*}
  By Lemma~\ref{lemma-surprise} and Theorem~\ref{new-mult-2},
  $\dpr(U_K) = \dpr(\Kk)$.  Because heaviness is the same as having
  full rank, every definable set $X \subseteq \Kk^\times$ containing
  $U_K$ is heavy.  Thus
  \begin{align*}
    & \bigcap \{X \div X : X \subseteq \Kk^\times \text{ heavy and
      $K$-definable}\} \\ \subseteq & \bigcap \{X \div X : X \subseteq
    \Kk^\times,~X \supseteq U_K,~X \text{ is $K$-definable}\}.
  \end{align*}
  As in the proof of Theorem~\ref{mult-1}.\ref{min-heavy}, this final
  intersection is $U_K$, because $U_K$ is a subgroup of $\Kk^\times$
  that is type-definable over $K$.  So all the inclusions above are
  equalities, and
  \begin{equation*}
    U_K = \{X \div X : X \subseteq \Kk^\times \text{ heavy and $K$-definable}\}.
  \end{equation*}
\end{proof}

\subsection{The key algebraic properties of the infinitesimals}
The following proposition summarizes the algebraic properties of the
infinitesimals that wil be important in the remaining sections.
\begin{proposition}\label{prop:prop}
  Let $\Kk$ be a saturated unstable dp-finite field, and let $K
  \preceq \Kk$ be small.
  \begin{enumerate}
  \item $I_K$ is a subgroup of $\Kk$.
  \item $I_K = I_K \cdot I_K$, where the right hand side means the set
    of finite sums
    \begin{equation*}
      x_1y_1 + \cdots + x_ny_n
    \end{equation*}
    with $x_i, y_i$ from $I_K$.
  \item $1 + I_K$ is a subgroup of $\Kk^\times$.  In particular, $-1
    \notin I_K$.
  \item For every $n \ge 1$, the $n$th power map
    \begin{align*}
      1 + I_K & \to 1 + I_K \\
      x & \mapsto x^n
    \end{align*}
    is onto.
  \item If $\characteristic(K) = p > 0$, then the Artin-Schreier map
    \begin{align*}
      I_K & \to I_K \\
      x & \mapsto x^p - x
    \end{align*}
    is onto.
  \end{enumerate}
\end{proposition}
\begin{proof}~
  \begin{enumerate}
  \item \cite{prdf}, Theorem~6.17
  \item The inclusion $I_K \cdot I_K \subseteq I_K$ holds by
    \cite{prdf}, Corollary~10.5---this is the reason why
    multiplication is continuous in the canonical topology.  For the
    reverse inclusion, let $J$ denote the true product
    \begin{equation*}
      J = \{xy : x, y \in I_K\}.
    \end{equation*}
    Then $J$ is a subset of $\Kk$, type-definable over $K$, of full
    dp-rank.  Therefore
    \begin{align*}
      J - J & = \bigcap \{X - X : X \supseteq J,~X \text{ is $K$-definable}\}
      \\ & \supseteq \bigcap \{X - X : X \text{ is $K$-definable and full rank}\}
      \\ & \supseteq \bigcap \{X \mininf X : X \text{ is $K$-definable and full rank}\}      
      \\ &= I_K.
    \end{align*}
    So $I_K \subseteq J - J \subseteq I_K \cdot I_K$.
  \item Theorems~\ref{mult-1}.\ref{m11} and \ref{new-mult-2}.
  \item The image is a subgroup of $\Kk^\times$, type-definable over
    $K$.  The $n$th power map has finite fibers, so the image has full
    dp-rank, and therefore contains $1 + I_K$ by
    Theorem~\ref{mult-1}.\ref{min-heavy} and
    Theorem~\ref{new-mult-1}.\ref{nm22}.
  \item Similar, using Corollary~6.19 in \cite{prdf}, which is the
    additive analogue of
    Theorem~\ref{mult-1}.\ref{min-heavy}. \qedhere
  \end{enumerate} 
\end{proof}

\section{Multi-valuation rings}\label{sec:mval-review}
We review well-known facts about Bezout domains and multi-valued
fields.
\begin{definition}
  A \emph{multi-valuation ring} on a field $K$ is a finite
  intersection of valuation rings on $K$.
\end{definition}
\begin{proposition} \label{multi-valuation-notes}
  Let $\Oo_1, \ldots, \Oo_n$ be pairwise incomparable valuation rings
  on a field $K$.  Let $\val_i : K \to \Gamma_i$ and $\res_i : \Oo_i
  \to k_i$ be the associated valuation and residue maps.  Let $\mm_i$
  be the maximal ideal of $\Oo_i$. Let $R = \bigcap_i \Oo_i$.
  \begin{enumerate}
  \item \label{mv1} Every finitely generated $R$-submodule of $K$ is
    singly generated.
  \item \label{mv2} In particular, $R$ is a Bezout domain.
  \item \label{mv3} $\Frac(R) = K$.
  \item \label{cut-description} Every $R$-submodule of $K$ is of the
    form
    \begin{equation*}
      \{ x \in K ~|~ \forall i : \val_i(x) > \Xi_i\}
    \end{equation*}
    for certain cuts $\Xi_i$ in the value groups $\Gamma_i$.  (The
    cuts $\Xi_i$ are not uniquely determined.)
  \item \label{mv5} Let $M_i = R \cap \mm_i$.  Then the $M_i$ are
    maximal ideals of $R$, the $M_i$ are pairwise distinct, and there
    are no other maximal ideals of $R$.
  \item \label{mv6} Furthermore, the canonical inclusion of $R/M_i
    \hookrightarrow k_i$ is an isomorphism.
  \item \label{mv7} The localization of $R$ at $M_i$ is exactly
    $\Oo_i$.
  \end{enumerate}
\end{proposition}
\begin{proof}
  We write $\res_i(x) = \infty$ if $x \notin \Oo$.
  \begin{claim}\label{u-claim}
    There are $u_1, \ldots, u_n \in K$ such that
    \begin{equation*}
      \res_i(u_j) = 
      \begin{cases}
        1 & \textrm{ if } i = j \\
        0 & \textrm{ if } i \ne j.
      \end{cases}
    \end{equation*}
  \end{claim}
  \begin{claimproof}
    We proceed by induction on $n$.  By symmetry we only need to find
    $u_1$.  For $n = 1$, take $u_1 = 1$.  For $n = 2$, by
    incomparability we can find $a \in \Oo_1 \setminus \Oo_2$ and $b
    \in \Oo_2 \setminus \Oo_1$.  Then
    \begin{align*}
      \val_1(a) & \ge 0 \\
      \val_1(b) & < 0 \\
      \val_2(a) & < 0 \\
      \val_2(b) & \ge 0.
    \end{align*}
    Therefore
    \begin{align*}
      \val_1(a/b) & > 0 \\
      \val_2(b/a) & > 0 \\
      \res_1(a/b) & = 0 \\
      \res_2(b/a) & = 0.
    \end{align*}
    Set $u_1 = b/(a + b)$.  Then
    \begin{align*}
      \res_1\left(\frac{b}{a + b}\right) &= \res_1\left(\frac{1}{a/b +
        1}\right) = \frac{1}{\res_1(a/b) + 1} = \frac{1}{0 + 1} = 1 \\
      \res_2\left(\frac{b}{a + b}\right) &= \res_2\left(\frac{b/a}{1 +
        b/a}\right) = \frac{\res_2(b/a)}{1 + \res_2(b/a)} = \frac{0}{1
        + 0} = 0.
    \end{align*}
    Finally, suppose $n > 2$. By the inductive hypothesis, we can find
    $v$ such that
    \begin{align*}
      \res_1(v) &= 1 \\
      \res_i(v) &= 0 \qquad \textrm{ for } 1 < i < n
    \end{align*}
    We may assume that $\res_n(v) \ne \infty$; otherwise replace $v$
    with $\frac{v}{v^2 - v + 1}$.  Similarly, we can find $w$ such that
    \begin{itemize}
    \item $\res_1(w) = 1$.
    \item $\res_{n-1}(w) \ne \infty$.
    \item $\res_i(w) = 0$ for $i \in \{2,3 \ldots, n-1,n-2,n\}$.
    \end{itemize}
    Then $u_1 = vw$ has the desired properties.
  \end{claimproof}
  \begin{claim}\label{bump-claim}
    For any $a \in K$ and any $S \subseteq \{1,\ldots,n\}$, we can
    find $a'$ such that
    \begin{align*}
      \val_i(a') &= \val_i(a) \qquad \textrm{ for } i \in S \\
      \val_i(a') &> \val_i(a) \qquad \textrm{ for } i \notin S
    \end{align*}
  \end{claim}
  \begin{claimproof}
    Indeed, take $a' = a \sum_{i \in S} u_i$.
  \end{claimproof}
  \begin{claim}
    For any $a, b \in K$, there is $c \in K$ such that $aR + bR = cR$.
  \end{claim}
  \begin{claimproof}
    We may assume $a, b$ are non-zero. By Claim~\ref{bump-claim} we
    may find $a' \in K$ such that for every $i$, one of the
    following happens
    \begin{align*}
      \val_i(a') = \val_i(a) &\ne \val_i(b) \\
      \val_i(a') > \val_i(a) &= \val_i(b)
    \end{align*}
    Thus for any $i$
    \begin{align*}
      \val_i(a') & \ge \val_i(a) \\
      \val_i(a') & \ne \val_i(b) \\
      \val_i(a' + b) &= \min(\val_i(a'),\val_i(b)) \\ & = \min(\val_i(a),\val_i(b)).
    \end{align*}
    It follows that $a' \in aR$, $a' + b \in aR + bR$, and $a, b \in
    (a'+b)R$.  Thus $aR + bR = (a'+b)R$.
  \end{claimproof}
  This proves points (\ref{mv1}) and (\ref{mv2}).  It also follows
  that $\Frac(R) = K$, point (\ref{mv3}).  Indeed, let $\alpha$ be any element
  of $K$.  Choose $c$ such that $R + \alpha R = cR$.  Then $c^{-1} R +
  c^{-1} \alpha R = R$, and so $c^{-1}$ and $c^{-1} \alpha$ are both
  in $R$; their ratio is $\alpha$.

  Point (\ref{cut-description}), the description of $R$-submodules of $K$, follows as
  well, essentially because every $R$-submodule of $K$ is a filtered
  union of singly-generated $R$-submodules, and a singly-generated
  $R$-submodule $aR \subseteq K$ is of the form
  \begin{equation*}
    \{ x \in K ~|~ \forall i : \val_i(x) \ge \val_i(a)\}
  \end{equation*}
  by definition of $R$.  More precisely, let $M$ be an $R$-submodule
  of $K$.  Let $\Xi_i$ be the largest cut in $\Gamma_i$ such that
  \begin{equation*}
    a \in M \implies \val_i(a) > \Xi_i
  \end{equation*}
  We claim that for any $x \in K$,
  \begin{equation*}
    x \in M \iff \forall i : \val_i(x) > \Xi_i.
  \end{equation*}
  The $\implies$ direction is by choice of $\Xi_i$.  Conversely,
  suppose the right hand side holds.  Choose $a_1, \ldots, a_n \in M$
  such that $\val_i(x) \ge \val_i(a_i)$.  Choose $b \in K$ such that
  \begin{equation*}
    bR = a_1R + \cdots + a_nR
  \end{equation*}
  Then $a_i \in bR \implies \val_i(a_i) \ge \val_i(b)$, and so
  \begin{equation*}
    \val_i(x) \ge \val_i(a_i) \ge \val_i(b).
  \end{equation*}
  As this holds for all $i$, the quotient $x/b$ lies in $R$, so $x \in
  bR$.  Then
  \begin{equation*}
    x \in bR = a_1R + \cdots + a_nR \subseteq M
  \end{equation*}
  so $x \in M$.  This proves point (\ref{cut-description}).
  \begin{claim}
    For each $i$, the map
    \begin{equation*}
      R \hookrightarrow \Oo_i \stackrel{\res_i}{\to} k_i
    \end{equation*}
    is onto.
  \end{claim}
  \begin{claimproof}
    Given $\alpha \in k_i$, we will find $a \in R$ such that
    $\res_i(a) = \alpha$.  We may assume $\alpha \ne 0$.  Take $c \in
    K$ such that $\res_i(c) = 1/\alpha$.  Let $S$ be the set of $j \ne
    i$ such that $\res_j c \ne \infty$.  Let $u_S = \sum_{j \in S}
    u_j$, so that
    \begin{equation*}
      \res_j(u_S) = 
      \begin{cases}
        0 & \textrm{ if } j = i \\
        0 & \textrm{ if } j \ne i \textrm{ and } \res_j(c) = \infty \\
        1 & \textrm{ if } j \ne i \textrm{ and } \res_j(c) \ne \infty.
      \end{cases}
    \end{equation*}
    Let $d = \frac{u_S}{1 - u_S}$, so
    \begin{equation*}
      \res_i(d) = \begin{cases}
        0 & \textrm{ if } j = i \\
        0 & \textrm{ if } j \ne i \textrm{ and } \res_j(c) = \infty \\
        \infty & \textrm{ if } j \ne i \textrm{ and } \res_j(c) \ne \infty.
      \end{cases}
    \end{equation*}
    Then
    \begin{equation*}
      \res_i(c + d) = \begin{cases}
        1/\alpha & \textrm{ if } j = i \\
        \infty & \textrm{ if } j \ne i \textrm{ and } \res_j(c) = \infty \\
        \infty & \textrm{ if } j \ne i \textrm{ and } \res_j(c) \ne \infty.
      \end{cases}
    \end{equation*}
    Therefore $1/(c + d)$ is in $R$ and $\res_i(1/(c + d)) = \alpha$.
  \end{claimproof}
  The kernel of the map $R \to k_i$ is exactly the set $M_i$, so we
  have shown that $M_i$ are maximal ideals and the canonical
  inclusions $R/M_i \hookrightarrow k_i$ are isomorphisms (point (\ref{mv6})).
  To complete point (\ref{mv5}), we must show that the $M_i$ are distinct and
  that there are no other maximal ideals.

  The $u_i$ of Claim~\ref{u-claim} show that the $M_i$ are pairwise
  distinct: note that $u_i \in M_j \iff i \ne j$.

  Suppose there is some maximal ideal $M_{n+1}$ outside of $\{M_1,
  \ldots, M_n\}$.  By the Chinese remainder theorem we can find $u \in
  R$ such that
  \begin{align*}
    u & \equiv 1 \pmod{M_i} \qquad i \le n \\
    u & \equiv 0 \pmod{M_{n+1}}.
  \end{align*}
  Then $\res_i(u-1) = 0$ for each $i \le n$, so $\res_i(u) = 1$ and
  $\res_i(u^{-1}) = 1$, implying that $u^{-1} \in R$.  But $u \in
  M_{n+1}$, a contradiction.  So $R$ has only the $n$ maximal ideals.
  This completes the proof of point (\ref{mv5}).

  Lastly, we show point (\ref{mv7}): the localization $R_{M_i}$ is exactly
  $\Oo_i$.  By definition,
  \begin{equation*}
    R_{M_i} = \left\{ \frac{a}{s}  :  a \in R, ~ s \in R \setminus M_i\right\}
  \end{equation*}
  Then $a \in \Oo_i$ and $s \in \Oo_i \setminus \mm_i$ by definition
  of $R$ and $M_i$.  So $a/s \in \Oo_i$, as $\Oo_i$ is a local ring.
  Therefore $R_{M_i} \subseteq \Oo_i$.

  Conversely, suppose $b$ is some element of $\Oo_i$.  As in the proof
  that $\Frac(R) = K$, we can find $a, s \in R$ such that
  \begin{align*}
    aR + sR &= R \\
    a/s &= b.
  \end{align*}
  The first equation implies that
  \begin{equation*}
    a\Oo_i + s\Oo_i = \Oo_i,
  \end{equation*}
  which means exactly that $0 = \min(\val_i(a),\val_i(s))$.  Now since
  $b \in \Oo_i$, we have
  \begin{equation*}
    \val_i(a) - \val_i(s) = \val_i(a/s) = \val_i(b) \ge 0,
  \end{equation*}
  and so
  \begin{equation*}
    0 = \min(\val_i(a),\val_i(s)) = \val_i(s),
  \end{equation*}
  which means exactly that $s \in R \setminus M_i$.  Therefore $b =
  a/s \in R_{M_i}$, completing the proof of (\ref{mv7}).
\end{proof}

\begin{corollary}\label{decomposition}
  If $R$ is a multi-valuation ring on a field $K$, then there is a
  unique way to write $R$ as a finite intersection of
  pairwise-incomparable valuation rings on $K$:
  \begin{equation*}
    R = \Oo_1 \cap \cdots \cap \Oo_n.
  \end{equation*}
  Moreover, these $\Oo_i$ are exactly the localizations of $R$ at its
  maximal ideals.
\end{corollary}
\begin{proof}
  Given $R = \Oo_1 \cap \cdots \cap \Oo_n$, we can discard irrelevant
  $\Oo_i$'s and arrange for the $\Oo_i$'s to be pairwise incomparable,
  proving existence.  Uniqueness follows from the ``moreover'' clause,
  which in turn follows from
  Proposition~\ref{multi-valuation-notes}.\ref{mv5},\ref{mv7}.
\end{proof}

\begin{corollary}\label{finite-chains}
  Let $\Oo_1, \ldots, \Oo_n$ be pairwise-incomparable valuation rings
  on $K$, and let $\Oo$ be some other valuation ring.  Then
  \begin{equation*}
    \Oo \supseteq \Oo_1 \cap \cdots \cap \Oo_n
  \end{equation*}
  if and only if $\Oo \supseteq \Oo_i$ for some $i$.
\end{corollary}
\begin{proof}
  Assume $\Oo \supseteq \Oo_1 \cap \cdots \cap \Oo_n$.  Then
  \begin{equation*}
    \Oo_1 \cap \cdots \cap \Oo_n = \Oo_1 \cap \cdots \cap \Oo_n \cap \Oo,
  \end{equation*}
  so by the uniqueness, the set $\{\Oo_1, \ldots, \Oo_n, \Oo\}$ cannot
  be pairwise incomparable.  Therefore $\Oo \supseteq \Oo_i$ or $\Oo
  \subseteq \Oo_i$.  If, say, $\Oo \subsetneq \Oo_1$ then
  \begin{equation*}
    \Oo_1 \cap \cdots \cap \Oo_n = \Oo_1 \cap \cdots \cap \Oo_n \cap
    \Oo = \Oo \cap \Oo_2 \cap \cdots \cap \Oo_n
  \end{equation*}
  again contradicting uniqueness.  So $\Oo \supseteq \Oo_i$ for some
  $i$.
\end{proof}

\begin{proposition}\label{alt-characterization}
  $R$ is a multi-valuation ring on $K$ if and only if
  \begin{enumerate}
  \item $\Frac(R) = K$.
  \item $R$ is a Bezout domain.
  \item $R$ has finitely many maximal ideals.
  \end{enumerate}
\end{proposition}
\begin{proof}
  If the conditions hold, then $R$ is a multi-valuation ring by
  \cite{prdf}, Remark~10.27.  Conversely, if $R$ is a multi-valuation
  ring, we may write $R = \Oo_1 \cap \cdots \cap \Oo_n$ where the
  $\Oo_i$ are pairwise incomparable.  Then the conditions hold by
  Proposition~\ref{multi-valuation-notes}.\ref{mv2},\ref{mv3},\ref{mv5}.
\end{proof}

\begin{proposition}\label{superring}
  Any superring of a multi-valuation ring is a multi-valuation ring.
\end{proposition}
\begin{proof}
  Let $R$ be a multi-valuation ring on $K$ and let $R'$ be a larger
  multi-valuation ring on $K$.  (So $R \subseteq R' \subseteq K$.)  By
  Proposition~\ref{multi-valuation-notes}.\ref{mv3},
  \begin{equation*}
    \Frac(R') \supseteq \Frac(R) = K,
  \end{equation*}
  so $\Frac(R') = K$.

  We claim that $R'$ is a Bezout domain.  Let $(a_1,\ldots,a_n)$ be a
  finitely generated ideal in $R'$.  By
  Proposition~\ref{multi-valuation-notes}.\ref{mv1}, there is some $b
  \in K$ such that
  \begin{equation*}
    a_1R + \cdots + a_nR = bR.
  \end{equation*}
  Thus, there are $c_1, \ldots, c_n$ and $d_1, \ldots, d_n$ in $R$
  such that
  \begin{equation*}
    b = a_1c_1 + \cdots + a_nc_n
  \end{equation*}
  \begin{equation*}
    a_i = d_ib.
  \end{equation*}
  Now the $c_i$ and $d_i$ are in $R'$, so
  \begin{equation*}
    a_1R' + \cdots + a_nR' = bR'.
  \end{equation*}
  Thus $R'$ is a Bezout domain with $\Frac(R') = K$.  Because $R'$ is
  a domain, $R'$ is the intersection of the localizations $R'_\mm$ at
  the maximal ideals $\mm$ of $R'$, and because $R'$ is a Bezout
  domain these localizations are valuation rings.  (These facts are
  elementary.  See \cite{prdf}, Remark~10.27, for example.)  So $R'$
  is the intersection of all valuation rings containing $R'$.

  We claim that the poset of valuation rings containing $R'$ is a
  finite union of chains.  It suffices to show the same fact for $R$,
  whose corresponding poset is bigger.  Writing $R = \Oo_1 \cap \cdots
  \cap \Oo_n$, we know by Corollary~\ref{finite-chains} that every
  valuation ring containing $R$ contains one of the $\Oo_i$.  For each
  $i$, the valuation rings containing $\Oo_i$ are the coarsenings of
  $\Oo_i$, which are totally ordered.  So the poset is a finite union
  of chains.

  Now the intersection of a chain of valuation rings on $K$ is a
  valuation ring on $K$, so it follows that $R'$ is a finite
  intersection of valuation rings.
\end{proof}

\begin{corollary}\label{lcming}
  Let $\Oo_1, \ldots, \Oo_n$ be pairwise incomparable valuation rings
  on $K$.  Let $\val_i : K \to \Gamma_i$ be the associated valuations.
  Given non-zero $a_1, \ldots, a_n$, there exist non-zero $b, c$ such
  that for every $i$,
  \begin{align*}
    \val_i(b) = \min(\val_i(a_1),\ldots,\val_i(a_n)) \\
    \val_i(c) = \max(\val_i(a_1),\ldots,\val_i(a_n)).
  \end{align*}
\end{corollary}
\begin{proof}
  We prove the existence of $b$; the existence of $c$ follows by
  replacing the $a_i$ with $a_i^{-1}$.  Let $R$ be the intersection
  $\bigcap_i \Oo_i$.  By Proposition~\ref{multi-valuation-notes}.1
  there is $b$ such that
  \begin{equation*}
    bR = a_1R + \cdots + a_nR.
  \end{equation*}
  Note that $b \ne 0$, because the right hand side is nonzero.  There
  exist $c_1, \ldots, c_n, d_1, \ldots, d_n \in R$ such that
  \begin{align*}
    b &= a_1c_1 + \cdots + a_nc_n \\
    a_i &= bd_i.
  \end{align*}
  The $c_i, d_i$ lie in $\Oo_j$, so
  \begin{equation*}
    b\Oo_j = a_1\Oo_j + \cdots + a_n\Oo_j
  \end{equation*}
  holds as well.  This implies exactly that
  \begin{equation*}
    \val_i(b) = \min(\val_i(a_1),\ldots,\val_i(a_n)). \qedhere
  \end{equation*}
\end{proof}

We shall need the following variant of the Chinese remainder theorem:
\begin{proposition}\label{prop:crt}
  Let $\Oo_1, \ldots, \Oo_n$ be pairwise incomparable valuation rings
  on $K$, with maximal ideals $\mm_i$ and residue fields $k_i$.  For
  any tuple $\vec{a} \in \prod_{i = 1}^n k_i$, there is $b \in
  \bigcap_{i = 1}^n \Oo_i$ such that
  \begin{equation*}
    b \equiv a_i \pmod{\mm_i}
  \end{equation*}
  for $i = 1, \ldots, n$.
\end{proposition}
\begin{proof}
  Let $R = \bigcap_{i = 1}^n \Oo_i$ and $M_i = \mm_i \cap R$.  By
  Proposition~\ref{multi-valuation-notes}.\ref{mv6}, we can find
  $c_1,\ldots,c_n \in R$ such that
  \begin{equation*}
    c_i \equiv a_i \pmod{\mm_i}.
  \end{equation*}
  By Proposition~\ref{multi-valuation-notes}.\ref{mv5}, the ideals
  $M_1,\ldots,M_n$ are pairwise distinct maximal ideals.  By the usual
  Chinese remainer theorem, there is $b \in R$ such that
  \begin{equation*}
    b \equiv c_i \pmod{M_i}
  \end{equation*}
  for each $i$.  Then
  \begin{equation*}
    b \equiv c_i \equiv a_i \pmod{\mm_i}
  \end{equation*}
  for all $i$.
\end{proof}

\section{Valuation-type dp-finite fields}\label{sec:val-type}
Let $\Kk$ be an unstable dp-finite field, let $K \preceq \Kk$ be a
small submodel, and $I_K$ be the group of $K$-infinitesimals.  We
have shown that the canonical topology on $K$ is a Hausdorff,
non-discrete field topology.  In terms of $I_K$, this means that
\begin{itemize}
\item $I_K \cap K = \{0\}$
\item $I_K \ne \{0\}$
\item $I_K \le (\Kk,+)$
\item $1 + I_K \le (\Kk,\times)$
\item $K \cdot I_K \subseteq I_K$.
\end{itemize}
By standard non-standard arguments, the following are equivalent:
\begin{itemize}
\item The canonical topology is a V-topology.  In other words, for any
  basic neighborhood $U \ni 0$, there is a basic neighborhood $V \ni
  0$ such that
  \begin{equation*}
    \forall x, y \in K : (x \notin U \wedge y \notin U \implies x
    \cdot y \notin V)
  \end{equation*}
\item $\Kk \setminus I_K$ is closed under multiplication.
\end{itemize}
By an algebraic exercise, the latter condition is in turn equivalent
to $I_K$ being the maximal ideal of a valuation ring on $\Kk$.

Recall from Corollary~\ref{alternative-bases} that the following
family of sets is a neighborhood basis of 0 in the canonical topology
on $K$:
\begin{equation*}
  \mathcal{F} = \{H(K) - H(K) : H \subseteq \Kk \textrm{ heavy and
    $K$-definable}\}.
\end{equation*}
\begin{lemma}\label{lem:v-top-uniformity}
  If the canonical topology on $K$ is a V-topology, then there is a
  uniformly definable family of heavy sets $\{H_b\}_{b \in Y}$ such
  that $\{(H_b - H_b)(K)\}_{b \in Y(K)}$ form a neighborhood basis of
  0.  The family $\{H_b\}_{b \in Y}$ can be chosen to be 0-definable.
\end{lemma}
\begin{proof}
  Let $\Oo_K$ be the valuation ring whose maximal ideal is $I_K$.
  Then $\Oo_K$ is $\vee$-definable over $K$ because
  \begin{equation*}
    \forall x \in \Kk^\times : x \in \Oo_K \iff x^{-1} \notin I_K.
  \end{equation*}
  By compactness, we can choose $K$-definable $B$ such that
  \begin{equation*}
    I_K \subseteq B \subseteq \Oo_K.
  \end{equation*}
  Shrinking $B$, we may assume $B = H - H$ for some $K$-definable
  heavy set $H$.  We claim that the family of sets $\{a \cdot
  B(K)\}_{a \in K^\times}$ is a neighborhood basis of 0.  Each set in
  this family is a neighborhood of 0, so it remains to show
  cofinality.  Let $B'$ be some other $K$-definable neighborhood of 0.
  Let $a \in \Kk$ be a nonzero $K$-infinitesimal.  Then
  \begin{equation*}
    a \cdot B \subseteq a \cdot \Oo_K \subseteq I_K \cdot \Oo_K \subseteq I_K \subseteq B'.
  \end{equation*}
  As $K \preceq \Kk$ and the sets $B, B'$ are $K$-definable, there is
  nonzero $a' \in K$ such that
  \begin{equation*}
    a' \cdot B \subseteq B'
  \end{equation*}
  and therefore
  \begin{equation*}
    (a' \cdot B)(K) = a' \cdot B(K) \subseteq B'(K).
  \end{equation*}
  This shows that the family of $\{a \cdot B(K)\}$ is a neighborhood
  basis.  This family can also be written as
  \begin{equation*}
    \{B_a(K)\}_{a \in K^\times}
  \end{equation*}
  where $B_a = (a \cdot H) - (a \cdot H)$.  As heaviness is
  0-definable, we can find a 0-definable family $\{H_b\}_{b \in Y}$ of
  heavy sets containing the sets $a \cdot H$ for every $a \in
  \Kk^\times$.  The family of sets $\{(H_b - H_b)(K)\}_{b \in Y(K)}$
  continues to be cofinal among neighborhoods of 0.
\end{proof}
\begin{definition}
  A dp-finite field $K$ has \emph{valuation type} if $K$ is unstable
  and the canonical topology is a V-topology.
\end{definition}
Unstable fields of dp-rank 1 have valuation type by
\cite{arxiv-myself}, Theorem~4.16.
\begin{proposition}\label{prop:val-type-equiv}
  If $K \equiv K'$ and $K$ is valuation type, then $K'$ is valuation
  type.
\end{proposition}
\begin{proof}
  Choose a 0-definable family $\{H_a\}_{a \in Y}$ of heavy sets such
  that in $K$, the family $\{(H_a - H_a)(K)\}_{a \in Y(K)}$ forms a
  neighborhood basis of 0.  By definition of V-topology, for any
  $K$-definable heavy set $H'$ there is $a \in Y(K)$ such that
  \begin{equation*}
    \forall x, y : (x \notin H' - H' \textrm{ and } y \notin H' - H')
    \implies x \cdot y \notin H_a - H_a.
  \end{equation*}
  This property of $K$ is expressed by a conjunction of first-order
  sentences, so it must hold in $K'$.  Therefore the canonical
  topology on $K'$ is also a V-topology.
\end{proof}

\subsection{From multi-valuation rings to valuation rings}\label{sec:toocool}
By Proposition~\ref{prop:prop}, the set of infinitesimals $I_K$ has
some unusual algebraic properties, which do \emph{not} hold for most
multi-valuation ideals.  Consequently, if $I_K$ is a multi-valuation
ideal, something special must be going on.

\begin{lemma}\label{5-way}
  Let $\Gamma$ be an ordered abelian group and let $\Xi$ be a cut in
  $\Gamma$.  Then at least one of the following holds:
  \begin{enumerate}
  \item \label{c1} There is an element $\gamma < 0$ such that for any
    $x \in \Gamma$,
    \begin{equation*}
      x > \Xi \implies x + \gamma > \Xi
    \end{equation*}
  \item \label{c2} There is an element $\gamma$ such that $\gamma <
    \Xi < 2 \gamma$
  \item \label{c3} There is an element $\gamma$ such that $2 \gamma <
    \Xi < \gamma$.
  \item \label{c4} $\Xi$ is the cut $0^+$
  \item \label{c5} $\Xi$ is the cut $0^-$.
  \end{enumerate}
\end{lemma}
\begin{proof}
  Assume none of \ref{c1},\ref{c4},\ref{c5} hold.  We break into two
  cases:
  \begin{itemize}
  \item $\Xi > 0$.  As $\Xi \ne 0^+$, there is some $0 < \gamma_0 <
    \Xi$.  As \ref{c1} fails to hold, there is some $x \in \Gamma$
    such that
    \begin{equation*}
      x > \Xi \text{ and } x - \gamma_0 < \Xi.
    \end{equation*}
    Changing variables, there is some $y \in \Gamma$ such that
    \begin{equation*}
      y < \Xi \text{ and } y + \gamma_0 > \Xi.
    \end{equation*}
    Then
    \begin{equation*}
      \max(y,\gamma_0) < \Xi < y + \gamma_0 \le 2 \max(y,\gamma_0),
    \end{equation*}
    so we may take $\gamma = \max(y,\gamma_0)$.
  \item $\Xi < 0$.  Similar. \qedhere
  \end{itemize} 
\end{proof}

\begin{proposition}\label{hostile-1}
  Let $R_0$ be a multi-valuation ring on a field $K$.  Let $J$ be an
  $R_0$-submodule of $K$, such that
  \begin{enumerate}
  \item $J \cdot J = J$
  \item $1 + J$ is a subgroup of $K^\times$ (in particular, $-1 \notin
    J$).
  \end{enumerate}
  Then $J$ is the Jacobson radical of a multi-valuation ring on $K$.
  Moreover, the multi-valuation ring is
  \begin{equation*}
    \{x \in K : xJ \subseteq J\}.
  \end{equation*}
\end{proposition}
\begin{proof}
  Let $R = \{x \in K : xJ \subseteq J\}$.  Then $R$ is a superring of
  $R_0$, so $R$ is a multi-valuation ring on $K$ by
  Proposition~\ref{superring}.  Write $R$ as an intersection of
  pairwise incomparable valuation rings:
  \begin{equation*}
    R = \Oo_1 \cap \cdots \cap \Oo_n
  \end{equation*}
  Let $(\val_i,\Gamma_i,\res_i,\mm_i,k_i)$ be the usual valuation data
  associated to $\Oo_i$.

  Let $\Xi_i$ be the largest cut in $\Gamma_i$ such that
  \begin{equation*}
    \forall x \in I_K : \Xi_i < \val_i(x)
  \end{equation*}
  As in the proof of
  Proposition~\ref{multi-valuation-notes}.\ref{cut-description}, we have
  \begin{equation*}
    x \in I_K \iff \forall i : \val_i(x) > \Xi_i.
  \end{equation*}
  Let $Q_i = \{x \in K : \val_i(x) > \Xi_i\}$.  Then $Q_i$ is an
  $\Oo_i$-submodule of $K$, and $J = \bigcap_{i = 1}^n Q_i$.  Let
  \begin{equation*}
    \Oo'_i = \{x \in K : xQ_i \subseteq Q_i\}.
  \end{equation*}
  Then $\Oo'_i$ is a superring of $\Oo_i$, hence a valuation ring.
  Let $R' = \bigcap_{i = 1}^n \Oo'_i$.  Note that if $x \in R'$ and $y
  \in J$, then
  \begin{equation*}
    x \in \bigcap_{i = 1}^n \Oo'_i \text{ and } y \in \bigcap_{i = 1}^n Q_i,
  \end{equation*}
  implying $xy \in \bigcap_{i = 1}^n Q_i = J$.  Thus
  \begin{equation*}
    R' \subseteq \{x \in K : xJ \subseteq J\} = R = \bigcap_{i = 1}^n
    \Oo_i \subseteq \bigcap_{i = 1}^n \Oo'_i = R'.
  \end{equation*}
  Thus $R = R'$.  By Corollary~\ref{finite-chains}, we see that for
  every $i$ there is $j$ such that
  \begin{equation*}
    \Oo_i \supseteq \Oo'_j \supseteq \Oo_j.
  \end{equation*}
  As the $\Oo_i$ were pairwise incomparable, $j$ must be $i$ and we
  conclude that $\Oo'_i = \Oo_i$ for all $i$.

  \begin{claim}
    Each $\Xi_i$ is the cut $0^+$ or $0^-$.
  \end{claim}
  \begin{claimproof}
    Use Lemma~\ref{5-way}.  We need to rule out cases
    \ref{c1}-\ref{c3}.
    \begin{enumerate}
    \item Suppose there is an element $\gamma < 0$ in $\Gamma_i$ such
      that for any $x \in \Gamma_i$,
      \begin{equation*}
        x > \Xi_i \implies x + \gamma > \Xi_i.
      \end{equation*}
      Take $a \in K$ with $\val_i(a) = \gamma$.  Then $aQ_i \subseteq
      Q_i$, so $a \in \Oo'_i = \Oo_i$, contradicting the fact that
      $\val_i(a) = \gamma < 0$.
    \item Suppose there is an element $\gamma \in \Gamma_i$ such that
      $\gamma < \Xi_i < 2 \gamma$.  By choice of $\Xi_i$, there is $a
      \in J$ such that $\val_i(a) \le 2 \gamma$.  As $J = J \cdot J$,
      we can write
      \begin{equation*}
        a = x_1 x_2 + \cdots + x_{2m-1}x_{2m}
      \end{equation*}
      for some $x_1,\ldots,x_{2m} \in J$.  But $\val_i(x_j) > \Xi_i >
      \gamma$, so $\val_i(x_{2j-1}x_{2j}) > 2 \gamma$.  Then $\val(a)
      > 2 \gamma$, contradicting the choice of $a$.
    \item Suppose there is an element $\gamma \in \Gamma_i$ such that
      $2 \gamma < \Xi_i < \gamma$.  By choice of $\Xi_i$, there is $a
      \in J$ such that $\val_i(a) \le \gamma$.  Then $a^2 \in J \cdot
      J = J$, but $\val_i(a^2) \le 2 \gamma < \Xi_i$, implying $a^2
      \notin J$.
    \end{enumerate}
    By Lemma~\ref{5-way}, the only other possibility is that $\Xi_i$
    is $0^\pm$.
  \end{claimproof}
  The claim means that each $Q_i = \{x \in K : \val_i(x) > \Xi_i\}$ is
  either the valuation ring $\Oo_i$ or the maximal ideal $\mm_i$.
  Without loss of generality,
  \begin{equation*}
    J = \Oo_1 \cap \cdots \cap \Oo_m \cap \mm_{m+1} \cap \cdots \cap
    \mm_n.
  \end{equation*}
  We claim that $m = 0$.  Otherwise, there is $a \in \bigcap_{i = 1}^n
  \Oo_i$ such that
  \begin{align*}
    a &\equiv -1 \pmod{\mm_1} \\
    a &\equiv 0 \pmod{\mm_j} \qquad (j > 1),
  \end{align*}
  by the Chinese remainder theorem (Proposition~\ref{prop:crt}).  Then
  $a \in J$, $1 + a \in 1 + J$, and so \[1/(1+a) \in 1 + J \subseteq
  \bigcap_{i = 1}^n \Oo_i.\] But $\res_1(1+a) = 0$, so $1/(1+a) \notin
  \Oo_1$, a contradiction.  Therefore $m = 0$ and
  \begin{equation*}
    J = \mm_1 \cap \cdots \cap \mm_n.
  \end{equation*}
  In other words $J$ is the Jacobson radical $R$.
\end{proof}

\begin{proposition}\label{prop:finally-a-valuation}
  Let $\Kk$ be a monster dp-finite unstable field and $K \preceq \Kk$
  be a small submodel.  Suppose $I_K$ is an $R$-submodule of $\Kk$ for
  some multi-valuation ring $R$ on $\Kk$.  Then $I_K$ is the maximal
  ideal of a valuation ring, so $K$ and $\Kk$ have valuation type.
\end{proposition}
\begin{proof}
  By Proposition~\ref{prop:prop}, $I_K$ satisfies the assumptions of
  Proposition~\ref{hostile-1}.  Therefore, we may change $R$ and
  assume that $I_K$ is the Jacobson radical of $R$:
  \begin{align*}
    I_K &= \mm_1 \cap \cdots \cap \mm_n \\
    R &= \Oo_1 \cap \cdots \cap \Oo_n,
  \end{align*}
  where the $\Oo_i$ are pairwise incomparable valuation rings and the
  $\mm_i$ are their maximal ideals.  Suppose for the sake of
  contradiction that $n > 1$.

  We first check that none of the $\Oo_i$ has mixed characteristic.
  Indeed, Proposition~\ref{hostile-1} ensures
  \begin{equation*}
    R = \{x \in \Kk : xI_K \subseteq I_K\}.
  \end{equation*}
  By Remark~6.9.3 in \cite{prdf}, $K \cdot I_K \subseteq I_K$.
  Therefore $K \subseteq R \subseteq \Oo_i$ for each $i$, ensuring
  that $\Oo_i$ is an equicharacteristic valuation ring.

  By Proposition~\ref{prop:prop}, we know the following:
  \begin{itemize}
  \item The squaring map on $1 + I_K$ is surjective.
  \item If $K$ has characteristic $p > 0$, the Artin-Schreier map on
    $1 + I_K$ is surjective.
  \end{itemize}
  We break into cases according to $\characteristic(\Kk)$.  First
  suppose $\characteristic(\Kk) \ne 2$.  By the Chinese remainder
  theorem (Proposition~\ref{prop:crt}), there is $a \in R$ such that
  \begin{align*}
    a &\equiv -1 \pmod{\mm_1} \\
    a &\equiv 1 \pmod{\mm_i} \qquad (i > 1),
  \end{align*}
  Then $\pm a \notin 1 + I_K$, but $a^2 \in 1 + I_K$, contradicting the
  surjectivity of the squaring map on $1 + I_K$.

  Next suppose $\characteristic(\Kk) = 2$.  By the Chinese remainder
  theorem, there is $a \in R$ such that
  \begin{align*}
    a &\equiv 1 \pmod{\mm_1} \\
    a &\equiv 0 \pmod{\mm_i} \qquad (i > 1),
  \end{align*}
  Then $a, 1 + a \notin I_K$, but $a^2 - a \in I_K$, contradicting the
  surjectivity of the Artin-Schreier homomorphism on $I_K$.

  In either case, we get a contradiction unless $n = 1$.  Thus $R =
  \Oo_1$ is a valuation ring and $I_K = \mm_1$ is its maximal ideal.
  Now $K$ has valuation type by the discussion at the start of
  \S\ref{sec:val-type}, and $\Kk$ has valuation type by
  Proposition~\ref{prop:val-type-equiv}.
\end{proof}

\section{Bounded groups}\label{sec:bdd}
In this section, we assume $\Kk$ is an unstable dp-finite field.
\begin{remark}\label{rem:heav}
  Let $n = \dpr(\Kk)$.  The following are equivalent for a
  type-definable subgroup $G \le (\Kk,+)$:
  \begin{enumerate}
  \item \label{ael1} $\dpr(G) = n$.
  \item \label{ael2} Every definable set $D \supseteq G$ has rank $n$.
  \item \label{ael3} Every definable set $D \supseteq G$ is heavy.
  \item \label{ael4} $G$ contains $I_K$ for some small subfield $K
    \preceq \Kk$.
  \end{enumerate}
\end{remark}
\begin{proof}
  The implication (\ref{ael1})$\implies$(\ref{ael2}) is trivial.  The
  equivalence (\ref{ael2})$\iff$(\ref{ael3}) is
  Theorem~\ref{new-mult-1}.\ref{nm22}.  The implication
  (\ref{ael3})$\implies$(\ref{ael4}) is Corollary~6.19 in \cite{prdf}.
  The implication (\ref{ael4})$\implies$(\ref{ael1}) is
  Lemma~\ref{lemma-surprise}.
\end{proof}

\begin{definition}
  Let $G \le (\Kk,+)$ be type-definable.
  \begin{itemize}
  \item $G$ is \emph{heavy} if it satisfies the equivalent conditions
    of Remark~\ref{rem:heav}.
  \item $G$ is \emph{bounded} if for every heavy subgroup $G' \le K$,
    there is non-zero $a \in \Kk$ such that $G \le a \cdot G'$.
  \end{itemize}
\end{definition}
Subgroups of bounded groups are bounded.  If $G$ is bounded and $a \in
\Kk$, then $a \cdot G$ is bounded.

\begin{definition}\label{def:magic}
  Let $\Kk$ be a dp-finite field.  A small submodel $K_0 \preceq \Kk$
  is \emph{magic} if for every type-definable subgroup $G \le
  (\Kk^n,+)$, we have
  \begin{equation*}
    K_0 \cdot G \subseteq G \implies G = G^{00}.
  \end{equation*}
  In other words, type-definable $K_0$-linear subspaces of $\Kk^n$ are
  00-connected.
\end{definition}
By \cite{prdf} (Theorem~8.4 and the proof of Corollary~8.7), all
sufficiently large submodels of $\Kk$ are magic.  In particular, magic
subfields exist.

Recall the notion of strict $r$-cubes and reduced rank (Definitions
9.13, 9.17 in \cite{prdf}).  A strict $r$-cube in a modular lattice
$M$ is an injective homomorphism of (unbounded) lattices from the
powerset of $r$ to $M$.  The \emph{base} of the cube is the image of
$\emptyset$ under this homomorphism.  The reduced rank $\redrk(M)$ is
the maximum $r$ such that a strict $r$-cube exists.  If $a \ge b$ are
elements of $M$, then $\redrk(a/b)$ is the reduced rank of the
sublattice $[b,a] \subseteq M$.

Fix a magic subfield $K_0$, and let $\Lambda = \Lambda_{K_0}$ be the
lattice of type-definable $K_0$-linear subspaces of $\Kk^1$.  Let $r =
\redrk(\Lambda)$ be the reduced rank of $\Lambda$.  The rank $r$ is
finite by \cite{prdf}, Proposition~10.1.7.
\begin{definition}\label{def:pedestal}
  A \emph{$K_0$-pedestal} is a group $J \in \Lambda$ that is the base
  of a strict $r$-cube in $\Lambda$, where $r = \redrk(\Lambda)$.
\end{definition}
In \S 10.1 of \cite{prdf}, $K_0$-pedestals were called ``special
groups.''

\begin{lemma}\label{lem:super-guards}
  Let $G \in \Lambda$ satisfy $\redrk(\Kk/G) = r$.  Then $G$ is
  bounded.
\end{lemma}
\begin{proof}
  We may assume $G$ is nonzero.  Let $H$ be any heavy subgroup.
  Choose a strict $r$-cube in the interval $[G,\Kk] \subseteq \Lambda$
  and let $J$ be the base of the cube.  Then $J$ is a $K_0$-pedestal.
  Let $K \preceq \Kk$ be a small submodel, chosen large enough that
  \begin{itemize}
  \item $I_K \subseteq H$ (Remark~\ref{rem:heav}).
  \item $G, J$ are type-definable over $K$.
  \item $K \supseteq K_0$.
  \end{itemize}
  The fact that $J$ is a $K_0$-pedestal implies $I_K \cdot J \subseteq
  I_K$, by \cite{prdf}, Proposition~10.4.3.  Thus
  \begin{equation*}
    I_K \cdot G \subseteq I_K \cdot J \subseteq I_K \subseteq H.
  \end{equation*}
  By Remark~6.9.1 in \cite{prdf}, there is nonzero $\epsilon \in I_K$.
  Then
  \begin{equation*}
    \epsilon \cdot G \subseteq I_K \cdot G \subseteq H,
  \end{equation*}
  verifying that $G$ is bounded.
\end{proof}
\begin{remark}\label{rem:pedestals-enclosed}
  In particular, $K_0$-pedestals are bounded.
\end{remark}

\begin{lemma}\label{wtf}
  Let $\{U_x\}$ be a 0-definable family of basic neighborhoods.  Then
  there is a 0-definable family of basic neighborhoods $\{V_x\}$ with
  the following property:
  \begin{equation*}
    \exists b ~\forall c ~\exists d : V_b \cdot V_d \subseteq U_c.
  \end{equation*}
\end{lemma}
\begin{proof}
  Fix a magic subfield $K_0$.  By Proposition~10.4.1 in \cite{prdf},
  there is a nonzero $K_0$-pedestal $J$.  Choose small $K_1 \succeq
  K_0$ such that $J$ is type-definable over $K_1$.  Choose $K_2
  \succeq K_1$ to be small but $|K_1|^+$-saturated.  Then
  \begin{align*}
    I_{K_1} & \subseteq J \\
    I_{K_2} \cdot J & \subseteq I_{K_2}
  \end{align*}
  by \cite{prdf}, Proposition~10.4.3.  Then for any $c \in \dcl(K_2)$,
  \begin{equation*}
    I_{K_2} \cdot I_{K_1} \subseteq I_{K_2} \cdot J \subseteq I_{K_2}
    \subseteq U_c.
  \end{equation*}
  Therefore there exists a $K_1$-definable neighborhood $W_1$ and a
  $K_2$-definable neighborhood $W_2$ such that $W_1 \cdot W_2
  \subseteq U_c$.  Because $K_2$ is $|K_1|^+$-saturated, we can choose
  $W_1$ from a finite family and $W_2$ from a definable family (of
  bounded complexity), independent of $c$.  Thus there exist
  $K_1$-definable basic neighborhoods $X_1, \ldots, X_n$ and a
  0-definable family $\{V_x\}$ of basic neighborhoods such that
  \begin{equation*}
    \forall c \in \dcl(K_2) ~\exists i \in \{1,\ldots,n\} ~\exists d
    \in \dcl(K_2) : X_i \cdot V_d \subseteq U_c.
  \end{equation*}
  Replacing $X_1, \ldots, X_n$ with their intersection, we may assume
  $X = X_i$ independent of $i$, and so
  \begin{equation*}
    \forall c \in \dcl(K_2)~ \exists d \in \dcl(K_2) : X \cdot V_d \subseteq U_c.
  \end{equation*}
  Enlarging $V_d$, we may assume $X = V_b$ for some $b \in \dcl(K_1)
  \subseteq \dcl(K_2)$.  Then
  \begin{equation*}
    \exists b \in \dcl(K_2) ~\forall c \in \dcl(K_2) ~\exists d \in
    \dcl(K_2) : V_b \cdot V_d \subseteq U_c.
  \end{equation*}
  This extends from $K_2$ to its elementary extension $\Kk$.
\end{proof}

\begin{lemma}\label{lem:overpower}
  If $K_1 \preceq K_2 \preceq \Kk$, then $I_{K_1} \cdot I_{K_2}
  \subseteq I_{K_2}$.
\end{lemma}
\begin{proof}
  Let $U$ be a $K_2$-definable basic neighborhood.  We can write $U$
  as $U_c$ for some 0-definable family of basic neighborhoods
  $\{U_x\}$ and some $c \in \dcl(K_2)$.  Let $\{V_x\}$ be a
  0-definable family of basic neighborhoods as in Lemma~\ref{wtf}, and
  let $b$ be such that
  \begin{equation*}
    \forall x ~\exists y : V_b \cdot V_y \subseteq U_x.
  \end{equation*}
  We can take $b \in K_1$, because $K_1 \preceq \Kk$.  Then, setting
  $x = c$ there must be some $d$ such that
  \begin{equation*}
    V_b \cdot V_d \subseteq U_b.
  \end{equation*}
  Moreover, we can take $d \in \dcl(K_2)$, because $b, c \in
  \dcl(K_2)$ and $K_2 \preceq \Kk$.  Now $V_b$ is a $K_1$-definable
  basic neighborhood and $V_d$ is a $K_2$-definable basic
  neighborhood, so
  \begin{equation*}
    I_{K_1} \cdot I_{K_2} \subseteq V_b \cdot V_d \subseteq U_b = U.
  \end{equation*}
  As $U$ was an arbitrary $K_2$-definable basic neighborhood, we
  conclude $I_{K_1} \cdot I_{K_2} \subseteq I_{K_2}$.
\end{proof}

\begin{corollary}\label{cor:inf-enclosed}
  For any small submodel $K$, the group $I_K$ is bounded.
\end{corollary}
\begin{proof}
  Take a magic subfield $K_0$.  By Proposition~10.4.1 in \cite{prdf},
  a nonzero $K_0$-pedestal $J$ exists.  Then $J$ is type-definable
  over some small model $K_1$ containing $K$ and $K_0$.  By
  Proposition~10.4.3 in \cite{prdf}, $I_{K_1} \subseteq J$.  As $J$ is
  bounded by Remark~\ref{rem:pedestals-enclosed}, it follows that
  $I_{K_1}$ is bounded.  Take a non-zero $\epsilon \in I_{K_1}$.  Then
  \begin{equation*}
    \epsilon \cdot I_K \subseteq I_{K_1} \cdot I_K \subseteq I_{K_1},
  \end{equation*}
  by Lemma~\ref{lem:overpower}.  Therefore $I_K$ is bounded.
\end{proof}

\begin{lemma}\label{lem:neo}
  Let $\Kk$ be an unstable dp-finite field.  Let $G \le (\Kk,+)$ be a
  bounded type-definable subgroup.  Suppose $G$ contains a non-zero
  $R$-submodule $M \le \Kk$ for some multi-valuation ring $R$ on
  $\Kk$.  Then $\Kk$ has valuation type.
\end{lemma}
\begin{proof}
  Choose a magic subfield $K_0$ and let $\Lambda = \Lambda_{K_0}$ be
  the lattice of type-definable $K_0$-linear subspaces of $\Kk$.  By
  Proposition~10.4.1 in \cite{prdf}, there is a nonzero $K_0$-pedestal
  $J$.  By Proposition~10.1.3 in \cite{prdf}, $J$ is heavy.  By
  definition of bounded, there is $a_1 \in \Kk^\times$ such that $a_1
  \cdot G \subseteq J$.  Replacing $G$ and $M$ with $a_1 \cdot G$ and
  $a_1 \cdot M$, we may assume $M \subseteq G \subseteq J$.  Let $a_2$
  be a non-zero element of $M$.  Then
  \begin{equation*}
    Ra_2 \subseteq M \subseteq G \subseteq J
  \end{equation*}
  and so
  \begin{equation*}
    R \subseteq a_2^{-1} \cdot J.
  \end{equation*}
  By Proposition~10.4.5 in \cite{prdf}, the group $a_2^{-1} \cdot J$
  is still a $K_0$-pedestal.  Let $K$ be a small field over which $a_2^{-1}
  \cdot J$ is type-definable.  By Proposition~10.4.3 in \cite{prdf},
  \begin{equation*}
    I_K \cdot R \subseteq I_K \cdot (a_2^{-1} \cdot J) \subseteq I_K.
  \end{equation*}
  Thus $I_K$ is an $R$-module, and $\Kk$ has valuation type by
  Proposition~\ref{prop:finally-a-valuation}.
\end{proof}

\begin{theorem}\label{thm:val-multival-type}
  Let $\Kk$ be a monster dp-finite unstable field.  The following are
  equivalent
  \begin{enumerate}
  \item \label{n1} $\Kk$ has valuation type.
  \item \label{n1.5} For every small submodel $K \preceq \Kk$, $K$ has
    valuation type.
  \item \label{n2} For every small submodel $K \preceq \Kk$, the
    infinitesimals $I_K$ are the maximal ideal of a valuation ring
    $\Kk$.
  \item \label{n3} For some small submodel $K \preceq \Kk$, the
    infinitesimals $I_K$ contain a nonzero $R$-submodule of $\Kk$, for
    some multi-valuation ring $R \subseteq \Kk$.
  \item \label{n4} Some bounded subgroup $G \le (\Kk,+)$ contains a
    nonzero $R$-submodule of $\Kk$, for some multi-valuation ring $R
    \subseteq \Kk$.
  \end{enumerate}
\end{theorem}
\begin{proof}
  The equivalence of (\ref{n1}) and (\ref{n2}) is
  Proposition~\ref{prop:val-type-equiv}.  The equivalence of
  (\ref{n1.5}) and (\ref{n2}) is discussed at the start of
  \S\ref{sec:val-type}.  The implication
  (\ref{n2})$\implies$(\ref{n3}) is trivial, because the
  infinitesimals are non-trivial and ideals are submodules.  Point
  (\ref{n3}) implies (\ref{n4}), by Corollary~\ref{cor:inf-enclosed}.
  Finally, the implication (\ref{n4})$\implies$(\ref{n1}) is
  Lemma~\ref{lem:neo}.
\end{proof}

We won't need the following fact, but it is nice to know conceptually:
\begin{proposition}
  If $G_1, G_2$ are two bounded type-definable subgroups of $(\Kk,+)$,
  then $G_1 + G_2$ is also bounded.
\end{proposition}
\begin{proof}
  Fix a magic field $K_0$.  By Proposition~10.4.1 in \cite{prdf},
  there is at least one non-zero $K_0$-pedestal $J$.  Then $J$ is
  heavy by \cite{prdf}, Proposition~10.1.3.  By definition of bounded,
  there are $a_1, a_2 \in \Kk^\times$ such that $G_1 \subseteq a_1
  \cdot J$ and $G_2 \subseteq a_2 \cdot J$.  By Proposition~10.4.5 in
  \cite{prdf}, scalings of $K_0$-pedestals are still $K_0$-pedestals,
  so $a_1 \cdot J$ and $a_2 \cdot J$ are $K_0$-pedestals.  Take a
  small model $K$ containing $K_0$ and type-defining the pedestals
  $a_1 \cdot J$ and $a_2 \cdot J$.  By Proposition~10.4.3 in
  \cite{prdf},
  \begin{align*}
    I_K \cdot (a_1 \cdot J) \subseteq I_K \\
    I_K \cdot (a_2 \cdot J) \subseteq I_K.
  \end{align*}
  Let $\epsilon$ be a non-zero $K$-infinitesimal.  Then
  \begin{align*}
    \epsilon \cdot G_1 & \subseteq I_K \cdot (a_1 \cdot J) \subseteq I_K \\
    \epsilon \cdot G_2 & \subseteq I_K \cdot (a_2 \cdot J) \subseteq I_K.
  \end{align*}
  Thus
  \begin{equation*}
    G_1 + G_2 \subseteq \epsilon^{-1} \cdot I_K.
  \end{equation*}
  By Corollary~\ref{cor:inf-enclosed}, it follows that $G_1 + G_2$ is
  bounded.
\end{proof}
As a consequence, bounded subgroups form an (unbounded) sublattice of
$\Lambda$.

\section{Finite extensions and henselianity}\label{sec:feh}
In this section, $\Kk$ is a saturated dp-finite field, no longer
assumed to be unstable.
\subsection{Finite extensions}\label{sec:9.1}

\begin{lemma}\label{lem:int-clos-algebra}
  Let $L/K$ be a finite separable extension.  Let $\Oo$ be a valuation
  ring on $K$ with maximal ideal $\mm$.  Choose some $K$-linear
  identification of $L$ with $K^d$, for $d = [L : K]$.  Then $\mm^d$
  contains a nonzero $R$-submodule of $L$ for some multi-valuation
  ring $R$ on $L$.
\end{lemma}
This is a variant of well-known facts in commutative algebra, but we
give the proof regardless.
\begin{proof}
  \begin{claim}
    There are valuation rings $\Oo_1, \ldots, \Oo_n$ on $L$, extending
    $\Oo$ on $K$, such that for any $x \in L$:
    \begin{equation*}
      x \in \bigcap_{i = 1}^n \Oo_i \implies \Tr_{L/K}(x) \in \Oo.
    \end{equation*}
  \end{claim}
  \begin{claimproof}
    Let $F$ be the normal closure of $L$ over $K$, so $F/L$ is finite
    and $F/K$ is Galois.  Let $\Oo'_1, \ldots, \Oo'_n$ enumerate all
    the extensions of $\Oo$ to $F$, and let $R'$ be the intersection
    of the $\Oo'_i$.  Then $R' \cap K = \Oo$, and $R'$ is fixed
    setwise by $\Gal(F/K)$.  Let $\sigma_1, \ldots, \sigma_d$ be the
    embeddings of $L$ into $F$ over $K$.  Each $\sigma$ can be
    extended to an automorphism of $F$, so
    \begin{equation*}
      x \in L \cap R' \implies \sigma_i(x) \in R'.
    \end{equation*}
    Now the trace of $x$ is by definition $\sum_{i = 1}^d
    \sigma_i(x)$, so we see
    \begin{equation*}
      x \in L \cap R' \implies \Tr_{L/K}(x) = \sum_{i = 1}^d
      \sigma_i(x) \in R'.
    \end{equation*}
    But $\Tr_{L/K}(x)$ is also in $K$, so
    \begin{equation*}
      x \in L \cap R' \implies \Tr_{L/K}(x) \in R' \cap K = \Oo.
    \end{equation*}
    Now $L \cap R' = \bigcap_{i = 1}^n \Oo_i$, where $\Oo_i$ is
    $\Oo'_i \cap L$, the restriction of $\Oo'_i$ to $L$.  Each $\Oo_i$
    is an extension from $K$ to $L$ of the original ring $\Oo$.
  \end{claimproof}
  Let $R$ be the intersection of the $\Oo_i$.  Let $e_1, \ldots, e_d$
  be the basis for $L$ over $K$ coming from the given identification
  of $L$ with $K^d$.  Because the extension $L/K$ is finite, the value
  group of $\Oo$ is cofinal in the value group of each $\Oo_i$.
  Therefore, by taking $\alpha \in K^\times$ with high enough
  valuation, we can ensure
  \begin{equation*}
    \{\alpha e_1 , \ldots, \alpha e_d \} \subseteq R.
  \end{equation*}
  Now for any $x \in R$ and $1 \le i \le d$, we have
  \begin{equation*}
    \Tr_{L/K}(x \alpha e_i) \in \Oo
  \end{equation*}
  by the claim.  Separability of $L/K$ implies that the trace pairing
  is non-degenerate.  The set of $\alpha e_i$ is a basis for $L$ over
  $K$.  Therefore, there is a $d \times d$ matrix $M_{ij}$ with
  entries from $K$ such that for any $x \in L$,
  \begin{equation*}
    x = \sum_{i = 1}^d \left(\sum_{j = 1}^d M_{ij} \Tr_{L/K}(x \alpha e_j)\right) e_i
  \end{equation*}
  We can choose $\beta \in K^\times$ such that $\beta M_{ij} \in \mm$
  for all $i, j$.  Then for any $x \in R$, we have
  \begin{equation*}
    \beta x = \sum_{i = 1}^d \sum_{j = 1}^d (\beta M_{ij})
    (\Tr_{L/K}(x \alpha e_j)) e_i \in \mm e_1 + \mm e_2 + \cdots + \mm
    e_d.
  \end{equation*}
  Thus $\mm^d = \mm e_1 + \cdots + \mm e_d$ contains the $R$-submodule
  $\beta R$.
\end{proof}

If $\Kk$ is some field, possibly with extra structure, and $\Ll$ is a
finite extension, we view $\Ll$ as a field with the following
structure:
\begin{itemize}
\item A unary predicate for the subfield $\Kk$.
\item The original structure on $\Kk$.
\end{itemize}
Up to naming finitely many parameters, the resulting structure is
bi-interpretable with $\Kk$, being interpreted as $\Kk^d$ for $d =
[\Ll : \Kk]$.  In particular, \[ \dpr(\Ll) = d \cdot \dpr(\Kk).\]

\begin{lemma}\label{power-bounded}
  Let $\Kk$ be an unstable dp-finite field.  Let $\Ll/\Kk$ be a finite
  extension of degree $d$.  Choose any $\Kk$-linear bijection $\Ll
  \cong \Kk^d$.
  \begin{enumerate}
  \item If $H$ is a heavy subgroup of $\Ll$, then $H \supseteq I_K^d$
    for some small subfield $K \preceq \Kk$.
  \item \label{pb2} If $G$ is a bounded subgroup of $\Kk$, then $G^d$
    is a bounded subgroup of $\Ll$.
  \end{enumerate}
\end{lemma}
\begin{proof}
  Naming parameters, we may assume that the identification of $\Ll$
  with $\Kk^d$ is 0-definable.
  \begin{enumerate}
  \item View $H$ as a subgroup of $\Kk^d$.  Take $K$ a small submodel
    over which $H$ is type-definable.  Take $a \in H$ with $\dpr(a/K)
    = \dpr(\Ll) = d \cdot \dpr(\Kk)$.  Write $a = (a_1,\ldots,a_d)$.
    Then
    \begin{equation*}
      \dpr(a_1/Ka_2,\ldots,a_d) = \dpr(\Kk),
    \end{equation*}
    by subadditivity of dp-rank.  Consider the type-definable sets
    \begin{align*}
      S &= \{x \in \Kk : (x,0,\ldots,0) \in H\} \\
      T &= \{x \in \Kk : (x,a_2,\ldots,a_d) \in H\}.
    \end{align*}
    Then $T$ is a coset of $S$, because $H$ is a group.  Thus $T = S +
    a_1$.  As $T$ is definable over $Ka_2,\ldots,a_d$, we have
    \begin{equation*}
      \dpr(S) = \dpr(T) \ge \dpr(a_1/Ka_2,\ldots,a_d) = \dpr(\Kk).
    \end{equation*}
    Thus $S$ is a heavy subgroup of $\Kk$.  As $S$ is $K$-definable,
    $I_K \subseteq S$.  This means that
    \begin{equation*}
      I_K \oplus 0^{d-1} \subseteq H.
    \end{equation*}
    A similar argument shows
    \begin{equation*}
      0^{i} \oplus I_K \oplus 0^{d-1-i} \subseteq H
    \end{equation*}
    for $0 \le i \le d$.  Therefore $I_K^d \subseteq H$.
  \item Let $H$ be any heavy subgroup of $\Ll$.  By the first point,
    there is $K \preceq \Kk$ such that $I_K^d \subseteq H$.  By
    Corollary~\ref{cor:inf-enclosed}, the group $I_K$ is bounded, so
    there is non-zero $a \in K$ such that $a \cdot G \subseteq I_K$.
    Then
    \begin{equation*}
      a \cdot (G^d) \subseteq a \cdot (I_K^d) \subseteq H,
    \end{equation*}
    proving that $G^d$ is bounded. \qedhere
  \end{enumerate} 
\end{proof}

\begin{theorem}\label{thm:finite-extensions}
  Let $\Kk$ be a field of valuation type.  Let $\Ll$ be a finite
  extension, viewed as a structure with a predicate for $\Ll$.  Then
  $\Ll$ has valuation type.
\end{theorem}
\begin{proof}
  Identify $\Ll$ with $\Kk^d$.  Take a small model $K \preceq \Kk$.
  Then $I_K$ is the maximal ideal of a valuation ring on $\Kk$, so
  $I_K^d$ contains a nonzero $R$-submodule of $\Ll$ for some
  multi-valuation ring $R$ on $\Ll$, by
  Lemma~\ref{lem:int-clos-algebra}.\footnote{The extension $\Ll/\Kk$
    is separable because $\Kk$ is perfect, which in turn holds by a
    trivial dp-rank calculation.}  By
  Lemma~\ref{power-bounded}.\ref{pb2} and
  Corollary~\ref{cor:inf-enclosed}, $I_K^d$ is a bounded subgroup of
  $\Ll$, and so $\Ll$ has valuation type by
  Theorem~\ref{thm:val-multival-type}.
\end{proof}

\subsection{Towards henselianity}
\begin{lemma}\label{non-independence}
  Suppose $\Kk$ is valuation type or stable.  If $\Oo_1$ and $\Oo_2$
  are two $\vee$-definable non-trivial valuation rings on $\Kk$, then
  they are not independent: $\Oo_1 \cdot \Oo_2 \ne \Kk$.
\end{lemma}
\begin{proof}
  The existence of non-trivial invariant valuation rings rules out the
  stable case, by Lemma~\ref{stable-valuation-rings}.  Let $K \preceq
  \Kk$ be a small submodel such that $\Oo_1$ and $\Oo_2$ are
  $\vee$-definable over $K$.  Let $R_K$ be the valuation ring whose
  maximal ideal is $I_K$.  Note that $R_K$ is $\vee$-definable over
  $K$ (it is $K$-invariant by construction, and it is $\vee$-definable
  because its maximal ideal $I_K$ is type-definable).  As
  non-independence is an equivalence relation on non-trivial valuation
  rings, it suffices to show that $\Oo_1$ and $R_K$ are
  non-independent.  Let $\mm_1$ be the maximal ideal of $\Oo_1$.  Then
  $\mm_1$ is type-definable over $K$.
  \begin{claim}
    $\dpr(\mm_1) = \dpr(\Kk)$.
  \end{claim}
  \begin{claimproof}
    As $\Oo_1$ is non-trivial, take non-zero $a$ with positive
    valuation.  Take $x \in \Kk$ with $\dpr(x/aK) = \dpr(\Kk)$.
    Replacing $x$ with $a \cdot x^{-1}$ if necessary, we may assume
    $x$ has positive valuation.  Then $x \in \mm_1$, and $\mm_1$ is
    type-definable over $aK$, so $\dpr(\mm_1) = \dpr(\Kk)$.
  \end{claimproof}
  For any $K$-definable $D \supseteq \mm_1$, we have $\dpr(D) =
  \dpr(\Kk)$, and so $D$ is heavy by
  Theorem~\ref{new-mult-1}.\ref{nm22} or \cite{prdf}, Lemma~7.1.  By
  Corollary~6.19 in \cite{prdf}, $I_{K} \subseteq \mm_1$.  It follows
  that
  \begin{equation*}
    R_K \supseteq \Oo_1,
  \end{equation*}
  so $R_{K}$ and $\Oo_1$ are non-independent.
\end{proof}

\begin{lemma}\label{comparability-finally}
  Let $\Kk$ be a dp-finite field.  Suppose that for every definable
  valuation ring $\Oo$ on $\Kk$ (possibly trivial), the residue field
  is stable or valuation type.  Then any two $\vee$-definable
  valuation rings $\Oo_1, \Oo_2$ on $\Kk$ are comparable.
\end{lemma}
\begin{proof}
  Suppose $\Oo_1, \Oo_2$ are incomparable.  By (easy) Remark 5.9 in
  \cite{arxiv-myself}, the join $\Oo_1 \cdot \Oo_2$ is a
  \emph{definable} valuation ring on $\Kk$.  If $\Kk'$ denotes the
  residue field of $\Oo_1 \cdot \Oo_2$, then
  \begin{itemize}
  \item $\Kk'$ is a dp-finite field, either valuation type or stable.
  \item $\Oo_1$ and $\Oo_2$ induce two independent non-trivial
    $\vee$-definable valuation rings on $\Kk'$.
  \end{itemize}
  This contradicts Lemma~\ref{non-independence}.
\end{proof}
The statement of Lemma~\ref{comparability-finally} is slightly
imprecise, because there is no canonical structure on the residue
field.  We should really say that the residue field is stable or
valuation-type with respect to \emph{any} amount of small induced
structure from $\Kk$.
\begin{lemma} \label{henselianity-finally}
  Suppose that $\Kk$ is a dp-finite field and for every definable
  valuation ring $\Oo$ on $\Kk$, the residue field is valuation type
  or stable.  Then any $\vee$-definable valuation ring $\Oo$ on $\Kk$
  is henselian.
\end{lemma}
\begin{proof}
  The assumption on $\Kk$ passes to finite extensions, because if
  $L/\Kk$ is finite and $\Oo$ is a definable valuation ring on $L$,
  then the residue field of $\Oo$ is a finite extension of the residue
  field of $\Oo \cap \Kk$.  So any two $\vee$-definable valuation
  rings on a finite extension of $\Kk$ must be comparable.  Now if
  $\Oo$ is a $\vee$-definable non-henselian valuation ring on $\Kk$,
  there is a finite extension $\Kk'/\Kk$ such that $\Oo$ has multiple
  extensions to $\Kk'$.  We may assume $\Kk'/\Kk$ is normal.  Let
  $\Oo_1, \Oo_2$ be two such extensions.  On general grounds, they
  must be incomparable.\footnote{The finite group $\Aut(\Kk'/\Kk)$
    acts transitively on the poset of extensions of $\Oo$ to $\Kk'$,
    so this poset must be totally incomparable.}  By Lemma 5.4 in
  \cite{arxiv-myself}, both $\Oo_1$ and $\Oo_2$ are $\vee$-definable.
  This contradicts Lemma~\ref{comparability-finally}.
\end{proof}

\begin{theorem} \label{thm:concluding-dichotomy}
  Suppose that $\Kk$ is a sufficiently saturated dp-finite field.
  Suppose that for every definable valuation ring $\Oo$ on $\Kk$, the
  residue field is valuation type or stable.  Then exactly one of the
  following happens:
  \begin{enumerate}
  \item $\Kk$ is stable and finite.
  \item $\Kk$ is stable and algebraically closed.
  \item $\Kk$ is unstable and admits a $K$-invariant non-trivial
    henselian valuation ring for some small $K \preceq \Kk$.
  \end{enumerate}
\end{theorem}
\begin{proof}
  If $\Kk$ is stable, then $\Kk$ has finite Morley rank by
  Corollary~\ref{cor:stable-finite-rank}.  If $\Kk$ is unstable, then
  by assumption $\Kk$ has valuation type (consider the trivial
  valuation).  Take a small model $K \preceq \Kk$.  Then $I_K$ is the
  maximal ideal of a non-trivial valuation ring $\Oo_K$ that is
  $\vee$-definable over $K$.  By Lemma~\ref{henselianity-finally}, it
  follows that $\Oo_K$ is henselian.
\end{proof}

\section{The multi-valuation strategy for classification}\label{retreat}
It is natural to make the following conjecture\footnote{Spoiler alert:
  this conjecture is probably false.  See
  \S\ref{spanner-in-the-works}.}:
\begin{conjecture}[valuation-type conjecture]\label{conj:vt}
  If $\Kk$ is a dp-finite field, then $\Kk$ is stable or has valuation
  type.
\end{conjecture}
By Theorem~\ref{thm:concluding-dichotomy}, this would imply the Shelah
and henselianity conjectures for dp-finite fields (see \S\ref{sec:1}),
which in turn yield a full classification of dp-finite fields
(\cite{halevi-hasson-jahnke}).

By Theorem~\ref{thm:val-multival-type} and
Corollary~\ref{cor:inf-enclosed}, Conjecture~\ref{conj:vt} is
equivalent to
\begin{conjecture}\label{conj:vt2}
  Let $\Kk$ be a saturated unstable dp-finite field.  Then there is a
  non-zero bounded subgroup $J \subseteq (\Kk,+)$ for which the
  ``stabilizer ring''
  \begin{equation*}
    R_J = \{x \in \Kk : xJ \subseteq J\}
  \end{equation*}
  is a multi-valuation ring on $\Kk$.
\end{conjecture}
Thus a natural strategy to classify dp-finite fields is to prove
Conjecture~\ref{conj:vt2}.  We call this \emph{the multi-valuation
  strategy}.
\begin{remark}
  Since the classification is already known for dp-finite fields of
  positive characteristic \cite{prdf}, we can restrict attention to
  fields of characteristic 0 in the above conjectures.
\end{remark}

\subsection{The dp-minimal case}
To motivate the multi-valuation strategy, we argue that it naturally
generalizes the strategy used in \cite{arxiv-myself} to classify
dp-minimal fields.  The strategy there can be outlined as follows:
\begin{enumerate}
\item \label{ftw1} Define a type-definable group $I_K$ of
  $K$-infinitesimals. (\cite{arxiv-myself}, \S 3)
\item \label{ftw2} Show that $1 + I_K$ is the set of multiplicative
  $K$-infinitesimals. (\cite{arxiv-myself}, proof of Proposition 4.13)
\item Use lattice-theoretic techniques to show that $I_K$ is an ideal
  in a valuation ring $\Oo_K$ on $\Kk$ (\cite{arxiv-myself} Theorem
  4.16, especially Claim 4.17).
\item Partially extend \ref{ftw1}, \ref{ftw2} to finite field
  extensions.  (\cite{arxiv-myself}, \S 5.1)
\item Consider the extensions of $\Oo_K$ to some finite field
  extension $\Ll/\Kk$.  Use the surjectivity of the squaring map $1 +
  I_L \to 1 + I_L$ to argue that $\Oo_K$ has a unique extension to
  $\Ll$.  (\cite{arxiv-myself}, proof of Proposition 5.6)
\item Conclude that $\Oo_K$ is henselian, and the Shelah conjecture
  holds.
\end{enumerate}
For dp-finite fields, the analogous steps are:
\begin{enumerate}
\item \label{ft1} Define a type-definable group $I_K$ of
  $K$-infinitesimals (done in \cite{prdf}).
\item \label{ft2} Show that $1 + I_K$ is the multiplicative
  $K$-infinitesimals (Theorem~\ref{new-mult-2}).
\item Somehow show that $I_K$ is an ideal in a multi-valuation ring
  $R_K$ on $\Kk$ ($\approx$ Conjecture~\ref{conj:vt2}).
\item Extend \ref{ft1}, \ref{ft2} to finite extensions of $K$.
  (Trivial, or \S\ref{sec:9.1}.)
\item Consider the extensions of $R_K$ to some finite extension
  $\Ll/\Kk$.  Use the surjectivity of the squaring map $1 + I_L \to 1
  + I_L$ to argue that $R_K$ is a valuation ring with a unique
  extension to $\Ll$.  (See the arguments of \S\ref{sec:toocool} and
  \S\ref{sec:bdd}-\ref{sec:feh}.)
\item Conclude that $R_K$ is a henselian valuation ring, and the
  Shelah conjecture holds.
\end{enumerate}

\subsection{Relation to \emph{Dp-finite fields I}}
The multi-valuation strategy was the secret motivation for much of
\cite{prdf}, \S 10.
\begin{itemize}
\item Section 10.1 of \cite{prdf} was concerned with $K_0$-pedestals,
  which are a special type of bounded group by
  Remark~\ref{rem:pedestals-enclosed}.
\item Section 10.2, specifically Proposition 10.15, developed
  techniques to study the stabilizer ring $R_J$ appearing in
  Conjecture~\ref{conj:vt2}.
\item Section 10.3, specifically Lemma 10.21, showed how to replace
  $J$ with a new pedestal $J'$ such that $R_{J'}$ is closer to being a
  multi-valuation ring.
\item This technique failed to terminate, but in Theorem 10.25 we
  obtained a multi-valuation ring $R^\infty$ as a filtered union of
  $R_J$'s.
\end{itemize}

\subsection{The problem} \label{spanner-in-the-works}
Unfortunately, Conjectures~\ref{conj:vt}-\ref{conj:vt2} \textbf{are
  probably false}.  I have carried out a detailed analysis of fields
of dp-rank 2 and characteristic 0, which will appear in subsequent
papers.  This analysis shows that if $\Kk$ has rank 2 and
characteristic 0, then one can find a pedestal $J$ for which the ring
$R_J$ is either:
\begin{itemize}
\item A valuation ring on $\Kk$.
\item An intersection of two valuation rings on $\Kk$.
\item A ring that is (essentially) of the form
  \begin{equation*}
    \{x \in \Kk : \val(x) \ge 0 \text{ and } \val(\partial x) \ge 0\}
  \end{equation*}
  for some valuation and derivation on $\Kk$.
\end{itemize}
I was unable to rule out the third case, which suggests the following
potential counterexample to Conjectures~\ref{conj:vt},\ref{conj:vt2}:
\begin{example}\label{dcvf}
  Let DCVF$_{0,0}$ denote the model companion of ACVF$_{0,0}$ expanded
  with a derivation.  (No compatibility between the derivation and
  valuation is assumed.)  Let $(K,v,\partial)$ be a model of
  DCVF$_{0,0}$.  Let $R$ be the following subring of $K$:
  \begin{equation*}
    R = \{x \in K : \val(x) \ge 0 \text{ and } \val(\partial x) \ge 0\}.
  \end{equation*}
  Consider the field $(K,+,\cdot,R)$ expanded by a unary predicate for
  $R$.  I suspect that this structure has rank 2, in which case it
  contradicts Conjectures~\ref{conj:vt}-\ref{conj:vt2}.
\end{example}
Nevertheless, two of the three cases in the analysis have valuation
type, and the results of the present paper apply.  Using variants of
\S\ref{sec:val-type}-\ref{sec:feh}, one can handle the third case, and
prove:
\begin{theorem}[to appear] \label{the-dream}
  Let $K = (K,+,\cdot,0,1,\ldots)$ be an expansion of a field.
  Suppose $K$ has dp-rank 2, characteristic 0, and is unstable.
  \begin{enumerate}
  \item Any two definable valuation rings on $K$ are comparable
  \item $K$ admits a definable V-topology
  \item The canonical topology on $K$ is definable.
  \end{enumerate}
\end{theorem}
This doesn't yet classify dp-rank 2 fields.  However, if
Theorem~\ref{the-dream} could be generalized to higher rank, it would
imply the henselianity and Shelah conjectures, completing the
classification of all dp-finite fields.

All this suggests that we should change
Conjectures~\ref{conj:vt}-\ref{conj:vt2}, replacing valuation rings
and multi-valuation rings with some larger class of ``meta-valuation
rings'' containing the ring $R$ of Example~\ref{dcvf}.  Then the
``meta-valuation strategy'' might succeed where the multi-valuation
strategy fails.

\begin{remark}
  One could distinguish two versions of Valuation
  Conjecture~\ref{conj:vt}:
  \begin{enumerate}
  \item \label{v1} The valuation conjecture for pure fields $(K,+,\cdot,0,1)$.
  \item \label{v2} The valuation conjecture for expansions of fields
    $(K,+,\cdot,0,1,\ldots)$.
  \end{enumerate}
  Version \ref{v1} should be true, given the expected classification of
  dp-finite fields.  On the other hand, version \ref{v2} is probably
  false because of the potential counterexample.

  Assuming the counterexample exists, any proof of the valuation
  conjecture for pure fields (version \ref{v1}) would need to use the
  purity assumption in an essential way.  But the purity assumption is
  very hard to use in proofs, unless the full classification is
  already known.
  
  Thus version \ref{v2} is probably false, and version \ref{v1} is
  probably unprovable without the classification in hand.
\end{remark}

\begin{acknowledgment}
The author would like to thank
\begin{itemize}
\item Martin Hils and Franziska Jahnke for the invitation to speak at
  the conference ``Model Theory of Valued Fields and Applications.''
  This paper began as supplementary notes for the talk.
\item Jan Dobrowolski, for some helpful discussions about groups of
  finite burden.
\item Meng Chen, for hosting the author at Fudan University, where
  this research was carried out.
\item Peter Sinclair and Silvain Rideau for information on DCVF.
\end{itemize}
{\tiny This material is based upon work supported by the National Science
Foundation under Award No. DMS-1803120.  Any opinions, findings, and
conclusions or recommendations expressed in this material are those of
the author and do not necessarily reflect the views of the National
Science Foundation.}
\end{acknowledgment}

\bibliographystyle{plain} \bibliography{mybib}{}

\end{document}